\documentclass[10pt]{amsart}
\usepackage{graphicx} 
\usepackage{color}
\usepackage[margin=2.8cm]{geometry}
\usepackage{hyperref}
\theoremstyle{plain} 
\newtheorem{theorem}{Theorem}[section]

\newtheorem{proposition}[theorem]{Proposition}
\newtheorem{claim}{Claim}
\newtheorem{fact}[theorem]{Fact}
\newtheorem{lemma}[theorem]{Lemma}
\newtheorem{corollary}[theorem]{Corollary}

\newtheorem{assumption}[theorem]{Assumption}

\DeclareMathOperator{\tr}{tr}
\DeclareMathOperator{\supp}{supp}
\DeclareMathOperator{\co}{co}
\DeclareMathOperator{\B}{B}
\DeclareMathOperator{\SO}{SO}
\DeclareMathOperator{\Ad}{Ad}

\title{Asymmetric stability of the Brunn--Minkowski inequality in compact Lie groups}
\author{Simon Machado}
\address{ETH Zurich, Ramistrasse 101, 8006 Zurich}

\begin{document}

\begin{abstract}
    We show a stability result for the recently established Brunn--Minkowski inequality in compact simple Lie groups. Namely, we prove that if two compact subsets $A, B$ of a compact simple Lie group $G$ satisfy 
    $$ \mu(AB)^{1/d'} \leq (1 + \epsilon)\left(\mu(A)^{1/d'} + \mu(B)^{1/d'}\right)$$
     where $AB$ is the \emph{Minkowski product} $\{ab : a \in A, b \in B\}$, $d'$ denotes the minimal codimension of a proper closed subgroup and $\mu$ is a Haar measure, then $A$ and $B$ must approximately look like neighbourhoods of a proper subgroup $H$ of codimension $d'$, with an error that depends quantitatively on $d', \epsilon$ and the ratio $\frac{\mu(A)}{\mu(B)}$. This result implies an improved error rate in the Brunn--Minkowski inequality in compact simple Lie groups
     $$\mu(AB)^{\frac{1}{d'}} \geq (1-C\mu(A)^{\frac{2}{d'}})\left(\mu(A)^{\frac{1}{d'}} + \mu(B)^{\frac{1}{d'}}\right) $$
     sharp, up to the constant $C$ which depends on $d'$ and $\frac{\mu(A)}{\mu(B)}$ alone.

     Our approach builds upon an earlier paper of the author proving the Brunn--Minkowski inequality, and stability in the case $A=B$. We employ a combinatorial multi-scale analysis and study so-called density functions. Additionally, the asymmetry between $A$ and $B$ introduces new challenges, requiring the use of non-abelian Fourier theory and stability results for the Prékopa–Leindler inequality.
\end{abstract}
\maketitle

\section{Introduction}

The Brunn--Minkowski inequality is a cornerstone of geometric analysis and convex geometry, establishing a deep relationship between the Lebesgue measures of subsets $A, B\subset \mathbb{R}^d$ and the measure of their \emph{Minkowski sum} $A+B:=\{a+b:a\in A,b\in B\}$. In its most well-known form, the inequality states:
\begin{equation}
    \lambda(A+B)^{1/d}\geq \lambda(A)^{1/d}+\lambda(B)^{1/d} \label{Eq: Brunn-Minkowski}
\end{equation}
where $\lambda$ denotes a Haar measure on $\mathbb{R}^d$. The Brunn--Minkowski inequality has profound implications and connections to a wide range of geometric problems, including the isoperimetric inequality, which has motivated extensive research in this area. Moreover, it is part of a broader family of inequalities that play a central role in various fields of mathematics. Notable members of this family include the Prékopa--Leindler inequality, the Brascamp--Lieb inequality, and the Sobolev inequality, among others. For a comprehensive overview, see \cite{BrunnMinkowski02Gardner}.

Equality cases for these inequalities have long been tied to convexity. A recent trend in the field has focused on the stability problem: given two sets that almost satisfy equality in \eqref{Eq: Brunn-Minkowski}, how far are they from sets that achieve equality? This question is not only natural but also crucial for applications, as it provides quantitative control over near-extremal cases. Initial progress on this problem for \eqref{Eq: Brunn-Minkowski} appeared in foundational works by Christ \cite{christ2012near} and Figalli–Jerison \cite{FigalliJerison15}, which have now led to sharp stability results for \eqref{Eq: Brunn-Minkowski}, as well as for the Prékopa–Leindler and Brascamp–Lieb inequalities due to Figalli--van Hintum--Tiba \cite{figalli2023sharp}. In this paper, we prove a stability result for the Brunn–Minkowski inequality in the context of compact simple Lie groups. 

Our main result is the following stability theorem: 

\begin{theorem}[Stability in compact simple Lie groups]\label{Theorem: Global stability}
Let $G$ be a compact simple Lie group of dimension $d$, let $g$ be the unique bi-invariant Riemannian metric and let $\mu$ be the Haar probability measure. Let $d'$ be the minimal co-dimension of a proper closed subgroup and $\tau, \delta > 0$.  Then there is $\epsilon > 0$ such that for all $A,B \subset G$ compact with $\mu(A), \mu(B) \leq \epsilon$, $\tau \leq \frac{\mu(A)}{\mu(B)} \leq \tau^{-1}$ and 
$$\mu(AB)^{\frac{1}{d'}} \leq (1+\epsilon)\left(\mu(A)^{\frac{1}{d'}} + \mu(B)^{\frac{1}{d'}}\right),$$
there are $g_1,g_2 \in G$, $\rho_1, \rho_2 > 0$ and a closed subgroup $H$ of codimension $d'$ such that $$ \mu(A \Delta g_1H_{\rho_1}) \leq \delta \mu(A)$$ and 
$$ \mu(B \Delta H_{\rho_2}g_2) \leq \delta \mu(B)$$
where $H_\rho$ denotes the collection of all points within distance $\rho$ of $H$. The dependence of $\delta$ on $\tau$ and $\epsilon$ can be computed.
\end{theorem}

While the stability problem of \eqref{Eq: Brunn-Minkowski} has been a central focus of recent research, the Brunn–Minkowski inequality’s influence extends far beyond its original setting in Euclidean spaces. Initially studied for convex subsets in $\mathbb{R}^d$, the inequality has since been generalized to a variety of settings, including Riemannian manifolds—where the Minkowski sum is replaced by a geodesic interpretation —as well as discrete settings and other ambient groups.  As demonstrated by our main result (Theorem \ref{Theorem: Global stability}), we focus in the present paper on extensions to other groups, a direction that has seen significant progress in recent years. Among the most notable classical results in this area are those of Kemperman \cite{Kemperman64} for tori, and the contributions of McCruden \cite{zbMATH03318704} and Tao \cite{BMnilpotentTao} in the context of simply connected nilpotent Lie groups.

Recent progress in extending Brunn–Minkowski-type inequalities to other ambient groups beyond isolated cases has been led by a series of papers \cite{zbMATH07715451, jing2023measure2, jing2023measure} by Jing, Tran, Zhang, and their co-authors.  Their work was driven by the potential relevance of the Brunn–Minkowski inequality to natural problems in additive combinatorics - an interaction already at the root of the works of Figalli--Jerrison mentioned above. Additive combinatorics focuses on the study of approximate multiplicative structures in groups, both discrete and continuous. A prime example of such structures is subsets $A$ whose product set $A^2:=\{ab: a,b \in A\}$ is not significantly larger than $A$ itself. Jing, Tran and Zhang observed that a tantalizing conjecture of Breuillard and Green - who asked whether, for all compact subsets $A \subset SO_3(\mathbb{R})$,
$$\mu_{SO_3(\mathbb{R})}(A^2) \geq (4 - o(\mu_{SO_3(\mathbb{R})}(A))) \mu_{SO_3(\mathbb{R})}(A)$$
- would fit naturally within the framework of generalizations of the Brunn–Minkowski inequality to compact Lie groups. Despite its apparent simplicity, this conjecture remained unsolved until \cite{jing2023measure}, where Jing, Tran and Zhang, leveraging insights from the Brunn–Minkowski inequality and model theory, successfully resolved it. However, their approach relied crucially on non-quantitative methods and did not extend to $SO_n(\mathbb{R})$ or any other compact groups. The full scope of the conjecture was later addressed by the author in \cite{Machado2024MinDoubling} using a quantitative method, leading to the following result:
\begin{equation}
    \forall A,B \subset G, \mu(AB)^{\frac{1}{d'}} \geq (1-\epsilon)\left(\mu(A)^{\frac{1}{d'}} + \mu(B)^{\frac{1}{d'}}\right) \label{eq: BM compact}
\end{equation} 
where $d'$ denotes the minimal codimension of a proper subgroup. An immediate consequence of our stability result (Theorem \ref{Theorem: Global stability}) is an improved estimate for the error rate in the Brunn--Minkowski inequality (\ref{eq: BM compact}), sharp up to a constant in $d'$ and $\frac{\mu(A)}{\mu(B)}$.

\begin{theorem}\label{Theorem: Error rate BM}
Let $G$ be a compact simple Lie group and $\tau > 0$. There is $C > 0$ such that for all $A,B \subset G$ compact with $\tau \leq \frac{\mu(A)}{\mu(B)} \leq \tau^{-1}$, $$\mu(AB)^{\frac{1}{d'}} \geq (1-C\mu(A)^{\frac{2}{d'}})\left(\mu(A)^{\frac{1}{d'}} + \mu(B)^{\frac{1}{d'}}\right) $$
where $d'$ denotes the minimal codimension of a proper subgroup and $\mu$ is a Haar measure.
\end{theorem}

Surprisingly, despite establishing strong quantitative stability results in this paper, the minimizers for the quantity $\mu(AB)$ remain unknown. This is due to the complexity of the group operation compared to addition, which we only understand up to a finite order via the Baker--Campbel--Hausdorff formula, see \eqref{Eq: BCH formula}. Theorem 1.1 strongly suggests that the minimizers are neighbourhoods of proper subgroups. If this where the case, it would imply $C\simeq \frac{S(2+\tau^{1/d'})\tau^{1/d'}}{6(d'+2)}$ in Theorem \ref{Theorem: Error rate BM} - where $S$ denotes the scalar curvature of the homogeneous space of $G$ of dimension $d'$ -  but the same dependence on $\mu(A)^{2/d'}$, further highlighting the significance of our results.
\subsection{Methods of proof}

The overall strategy for proving Theorem \ref{Theorem: Global stability} builds on the approach developed in \cite{Machado2024MinDoubling}. In that work, we established a stability result for the doubling inequality
$$ \mu(A^2) \geq \left(2^{d'} - o(\mu(A))\right) \mu(A).$$
which corresponds to the special case $A=B$ of Theorem \ref{Theorem: Global stability}. Briefly, the proof proceeds in three main steps. First, we use a key result from \cite{Machado2024MinDoubling} to reduce the problem to the case where $A$ and $B$ are contained in a small neighbourhood of a proper subgroup. Then, a combination of a double-counting argument (see \S \ref{Subsection: Double counting}) and a reconstruction procedure based on Ulam stability for compact groups (see \S \ref{Subsection: Concluding}) allows us to further reduce the problem to a local stability result. Under the symmetry assumption from \cite{Machado2024MinDoubling}, the local result is then obtained using a general limit argument and a simple combinatorial lemma, both of which rely on the equality case in Euclidean spaces.

In the present paper, we refine and extend these ideas. We make the limit argument quantitative and have to completely modify the combinatorial lemma to handle the asymmetric case. The intuition behind the local stability result is as follows: close to the identity, Lie groups resemble their Lie algebra. Thus, locally, a stability result in a Lie group should resemble stability results in Euclidean space of the same dimension. Because of this, we expect convex subsets of the Lie algebra to play a central role. This leads to the following local stability result:

\begin{theorem}\label{Theorem: Local stability}
    Let $G$ be a compact simple Lie group of dimension $d$ and $\tau, \rho,\epsilon > 0$. Let $g$ denote the unique Riemannian bi-invariant metric on $G$ with diameter $1$ and $\B(e,\rho)$ denote the ball of radius $\rho$ at $e$. Suppose that $A,B \subset \B(e,\rho)$ are two compact subsets such that $$\tau^{-1}\mu(B) \leq \mu(A) \leq \tau \mu(B)$$
    and 
    $$ \mu(AB) \leq (1+\epsilon)\left(\mu(A)^{1/d}+ \mu(B)^{1/d}\right)^d.$$ 
    Then there is a convex subset $C$ in the Lie algebra of $G$ such that $\exp( C) \subset \B(e,\rho)$ contains $A$ and
    $$\mu(\exp(C)\setminus A) = O_{d,\tau}(\epsilon + \rho)^c \mu(\B(e,\rho)), \hspace{1cm} \mu(\exp(\lambda C)\Delta B) = O_{d,\tau}(\epsilon + \rho)^c \mu(\B(e,\rho))$$
    and 
    $$ \mu\left(\exp\left((1+\lambda\right)C) \Delta AB\right)= O_{d,\tau}(\epsilon + \rho)^c \mu(\B(e,\rho)) $$
    for some $c \gg_{d,\tau} 1$ and where $\lambda = \frac{\mu(B)^{1/d}}{\mu(A)^{1/d}}$.
\end{theorem}

Just like Theorem \ref{Theorem: Global stability} provides quantitative control over near extremal sets relative to the global Brunn--Minkowski inequality, Theorem \ref{Theorem: Local stability} is the corresponding statement pertaining to the local Brun--Minkowski inequality
\begin{equation}
    \forall A,B \subset \B(e,\rho) \text{ compact},\    \mu(AB)^{1/d} \geq (1- O_d(\rho^2))\left(\mu(A)^{1/d} + \mu(B)^{1/d}\right) \label{Eq: Local BM}
\end{equation} 
proved in \cite{Machado2024MinDoubling} using optimal transport.

Compared to the symmetric case, Theorem \ref{Theorem: Local stability} is significantly more involved. The asymmetry of $A$ and $B$ requires us to leverage many of the tools introduced in \cite{Machado2024MinDoubling} (e.g., density functions and multi-scale analysis) as well as exploit additional techniques, such as non-abelian Fourier analysis and stability results for the Prékopa–Leindler inequality (Proposition \ref{Proposition: Application PL stability}).

First, by the Baker–Campbell–Hausdorff formula, group multiplication in the Lie group corresponds to addition in the Lie algebra up to second-order terms (Corollary \ref{Corollary: Sum vs product}). We combine this with stability results for the Prékopa--Leindler inequality to show that $A\B(e,\rho^2)$  is approximately convex — more precisely, that the density function of $A$ at scale $\rho^2$ is log-concave. 

Secondly, we prove that there is a constant $c$ such that if $A\B(e,\rho^c)$ has large Haar measure, then $A$ itself must have large Haar measure (Proposition \ref{Proposition: Small doubling and large cov number implies large measure}). This result is a specificity of simple Lie groups and relies on multi-scale toolsets for simple Lie groups from \cite{zbMATH06466329} and \cite{Machado2024MinDoubling}. At its core, it exploits a local nilpotent-like behaviour of commutator in Lie groups and the fact that sets with small doubling behave like groups at all but finitely many scales, a fact already key in \cite{Machado2024MinDoubling}. 

Finally, we rely on Non-abelian Fourier theory to show that products of two subsets of high density in $\B(e,\rho)$ must contain balls with radius polynomial in $\rho$ (Proposition \ref{Proposition: Quasi-random}). This property - without the polynomial dependence on the radii - was first coined \emph{local quasi-randomness} by Gowers and Long in \cite{zbMATH07286399} and already used in \cite{zbMATH07599280}. Here, we follow the same proof principle as \cite[App. A]{zbMATH07286399} while carefully tracking polynomial error rates. 

\subsection{Outline of the paper}
We start in \S \ref{Preliminaries} with some preliminaries. We state the Baker--Campbell--Hausdorff formula and some of its consequences, we introduce the Prékopa--Leindler inequality and a stability result, we establish a number of useful facts on convex subsets and we provide necessary background on non-abelian Fourier theory. 

In \S \ref{Local stability} we then prove the local stability result (Theorem \ref{Theorem: Local stability}). We rely on the study of density functions and approximate groups, both concepts being defined there. 

Finally, in \S \ref{Global stability} we prove the global stability result (Theorem \ref{Theorem: Global stability}). This part resembles \cite{Machado2024MinDoubling} most, but we track down all estimates to obtain quantitative results. We also pinpoint the only non-polynomial quantitative result involved. 
\subsection*{Acknowledgement} The author gratefully acknowledges the support of the \emph{Forschungsinstitut für Mathematik} (FIM). 
\section{Preliminaries}\label{Preliminaries}
\subsection{Notation}
Given two subsets $X,Y$ of a group $G$, $XY$ denotes the \emph{Minkowski product} $\{xy : x \in X, y \in Y\}$. We define inductively $X^1=X$ and $X^{n+1}=XX^n$ for all non-negative integers $n$. Finally, $X^{-1}:=\{x^{-1} | x \in X\}$ and $X^{-n}:=\left(X^{-1}\right)^n$.

For $f,g: G \rightarrow \mathbb{R}_{\geq 0}$ two functions, we write $f \ll_{x,y,...} g$ and $f = O_{x,y,\ldots}(g)$ to signify $|f| \leq C g $ where $C>0$ is a constant depending on the parameters $x,y,\ldots$ alone. 

\subsection{Compact Lie groups and Lie algebras}

Throughout, $G$ denotes a connected compact simple Lie group with Lie algebra $\mathfrak{g}$ and $d$ denotes the unique bi-invariant Riemannian distance which gives diameter $1$ to $G$. The probability Haar measure is denoted by $\mu$. We denote by $\B(g,\rho) \subset G$ the ball centered at $g$ of radius $\rho$. The ball centered at identity is \emph{normal} i.e. for all $g \in G$, $g\B(e,\rho)g^{-1}=\B(e,\rho)$ because $d$ is bi-invariant. For a subset $X \subset G$ and $\rho >0$, we write $X_\rho:=X\B(e,\rho)$. 

Near the identity, the group multiplication and addition are close to one another. This is expressed by the Baker--Campbell--Hausdorff formula:

\begin{proposition}[Baker--Campbell--Hausdorff formula, \S 5.6, \cite{LieGroups15Hall}]\label{Proposition: BCH formula}
There is a neighbourhood $V$ of  $0 \in \mathfrak{g}$ such that for every $X,Y \in V$ we have 
\begin{equation}
    \log \left(\exp X \exp Y\right) = X + Y + \frac12 [X,Y] + \ldots \label{Eq: BCH formula}
\end{equation}
where the remaining terms range over all higher Lie brackets made using $X$ and $Y$ and the coefficients involved are universal.
\end{proposition}

The main consequence we will use is the following: 

\begin{corollary}\label{Corollary: Sum vs product}
    Let $\rho > 0$, $x,y \in \B(e,\rho)$. Then 
    $$ x+y \in \B\left(xy,O_d(\rho^2)\right).$$
    In particular, 
    $$ yx \in \B\left(xy,O_d(\rho^2)\right).$$

\end{corollary}

Another useful consequence concerns the relation between the Lebesgue measure on the Lie algebra and the Haar measure.

\begin{fact}[e.g. Prop. 3.2,\cite{Machado2024MinDoubling}]\label{Fact: Lebesgue vs Haar}
    There is a Lebesgue measure $\lambda$ on the Lie algebra $\mathfrak{g}$ of $G$ and $\rho_0 > 0$ such that for all $\rho_0 \geq \rho > 0$ and $A \subset \log \B(e,\rho)$, $\log A$ is well-defined and 
    $$ \lambda(\log A) = (1 + O_d(\rho^2)) \mu(A).$$
In particular, we have
    $$ \mu(\B(e,\rho)) = vol(\rho) + O_d\left( \rho^2 vol(\rho)\right)$$
    where $vol(\rho)$ denotes the volume of the ball of radius $\rho$ in $\mathbb{R}^d$.
\end{fact}

\subsection{Facts about convex sets}

We will use the following well-known fact from convex geometry: 

\begin{fact}[John Ellipsoid, e.g. \cite{zbMATH01149836}]\label{Fact: John Ellipsoid}
Let $C \subset \mathbb{R}^d$ be a convex subset. We call the ellipsoid $E \subset C$ of maximal volume the \emph{John ellipsoid}. If $c$ is the center of $E$, then 
$$C \subset c + d(E-c):=\{c + d(x-c) | x \in E\}.$$
In particular, if $C$ is symmetric, then $c=0$.
\end{fact}

Recall that an ellipsoid in $\mathbb{R}^n$ is simply the image of a Euclidean ball under an affine map $x \mapsto Ax + y$ with $A$ definite-positive and symmetric. We will refer to any eigendirection of $A$ as an \emph{axis} of the associated ellipsoid. The \emph{length} of an axis is the associated eigenvalue.

\begin{lemma}\label{Lemma: Volume of thick convex sets}
    Let $1<c$. Let $C\subset \mathbb{R}^d$ be a convex subset contained in the ball $\B(0,\rho)$ of radius $\rho$ at $0$. Then for $\rho$ sufficiently small
    $$ \mu(C+ \B(e,\rho^c)) \ll_d \mu(L) + \rho^{d+ (c-1)}.$$
\end{lemma}

\begin{proof}
     Let $l_1 \geq  \ldots \geq l_d$ be the lengths of the axes of the John ellipsoid of $C$ ordered by size (i.e. the eigenvalues of the associated positive-definite symmetric matrix). Note that $l_1 \leq \rho$. Then $$\lambda(C+ \B(e,\rho^{c})) \ll_d \prod(l_i+\rho^c) \ll_d \mu(C) + \rho^{d-1+c}.$$
\end{proof}

Lemma \ref{Lemma: Volume of thick convex sets} is particularly attractive when used in combination with the following. 

\begin{lemma}\label{Lemma: Small sym diff implies close}
    Let $C_1,C_2\subset \mathbb{R}^d$ be convex subsets contained in the ball $\B(0,1)$ of radius $1$ at $0$. Assume that $\lambda(C_1) \geq \eta$. If $\mu(C_1 \Delta C_2) \leq \delta \mu(C_1)$, then $C_2 \subset C_1 + \B(0,\delta^c)$ for some $c \gg_{d,\eta} 1$. 
\end{lemma}

\begin{proof}
    The convex set $C_1 \cap C_2$ has measure at least $\eta$ and is contained in $\B(0,1)$. So Fact \ref{Fact: John Ellipsoid} implies that it contains a ball of radius $\rho_0 \gg_{d,\eta} 1$. Now, for any point $x \in \B(0,1)$ and $0 \leq \rho \leq 1$, the convex set $C_{x,\rho}$ defined as the intersection of the convex hull of $\{x\} \cup \B(0,\rho_0)$ and $\B(x,\rho)$ has measure $\gg_{d,\eta} \rho^{c}$ for some $c \gg_{d,\eta} 1$. Hence, for some $\rho \ll_{d,\eta} (\delta\mu(C))^c$ we have $\mu(C_{x,\rho}) > \delta \mu(C_1) \geq \mu(C_1 \Delta C_2)$.
    
    Let $x \in C_1 \setminus C_2 + \B(0,\rho)$. We have that $C_{x,\rho} \subset C_1 \setminus C_2$. So $\mu(C_1 \setminus C_2) > \mu(C_1 \Delta C_2 ) $. A contradiction.
 
\end{proof}
We record also two additional facts which will be crucial at the very end of this paper: 

\begin{lemma}\label{Lemma: center of mass}
    Let $C \subset \mathbb{R}^d$ be a convex subset contained in a ball of radius $1$ with center of mass $0$. In other words, 
    
    $$ \int_C x dx = 0.$$ Suppose also that $i: \mathbb{R}^n \rightarrow \mathbb{R}^n$ is an affine isometry of $\mathbb{R}^n$ such that 
    $$\lambda(C \Delta i(C)) \leq \eta.$$
    Then $|i(0) - 0 | \leq \eta$.
\end{lemma}

\begin{proof}
    We have 
    \begin{align*}
        \frac{|i(0) - 0 |}{\lambda(C)} &= \left|\frac{i\left(\int_C x dx\right)}{\lambda(C)} - \frac{\int_Cxdx}{\lambda(C)} \right| \\ 
        & = \frac{\left|\int_C i(x) dx- \int_Cxdx \right|}{\lambda(C)} \\
        &= \frac{\left|\int_{i(C)} x dx- \int_Cxdx \right|}{\lambda(C)} \\
        &\leq \frac{\mu(i(C) \Delta C)}{\lambda(C)} \leq \frac{\eta}{\lambda(C)}.
    \end{align*}
\end{proof}

In addition: 

\begin{lemma}\label{Lemma: Almost symmetries and symmetries}
    Let $\delta, \eta >0$. Let $0 \in C \subset  \mathbb{R}^d$ be a convex subset of measure at least $\eta$ contained in the ball $\B(0,1)$ of radius $1$ at $0$. Suppose that $\Gamma$ is a subgroup of $\SO_m(\mathbb{R})$ such that for all $\gamma \in \Gamma$
    $$ \mu\left( \gamma(C) \Delta C \right) \leq \delta \mu(C).$$
    Then there is $C' \supset C$ convex and invariant under $\Gamma$ such that $\mu(C' \setminus C) \ll_{d,\eta} \delta^c \mu(C)$ for some $c \gg_{d,\eta} 1$.
\end{lemma}

\begin{proof}
According to Lemma \ref{Lemma: Small sym diff implies close}, $\gamma(C) \subset C + \B(e,O_{d,\eta}(\delta)^c)$ for all $\gamma$. If $C'$ denotes the convex hull of $\bigcup \gamma(C)$, then $C \subset C' \subset  C + \B(e,O_{d,\eta}(\delta)^c)$. So we conclude by Lemma \ref{Lemma: Volume of thick convex sets}. 
\end{proof}

 And a counterpart when no lower bound on the measure is assumed:
 
\begin{lemma}\label{Lemma: Symmetric short axis implies small volume}
    Let $H \subset \SO_d(\mathbb{R})$ be a closed group of isometries of $\mathbb{R}^d$ acting irreducibly. There is $\delta > 0$ such that for any $X \subset H$ with $X\B_{\SO_d(\mathbb{R})}(e,\delta) = H$ where $B_{\SO_d(\mathbb{R})}(e,\delta)$ denotes the ball at $e$ of radius $\delta$ with respect to the Riemannian bi-invariant metric of diameter $1$, the following holds. Let $C \subset \mathbb{R}^d$ be a compact convex subset. Suppose that one of the axes of $C$ has length $l$. Then
    $$ \bigcap_{h \in X} h(C)$$
    has diameter $O_{d}(l)$.
\end{lemma}

\begin{proof}
    By assumption, there is a unit vector $v_0 \in \mathbb{R}^d$ such that 
    $C \subset v_0^{\perp} + \B(c, dl)$ - where $c$ is the center of the John ellipsoid of $C$ - 
    according to Fact \ref{Fact: John Ellipsoid}. Because $H$ acts irreducibly, there are $h_1, \ldots, h_d$ such that the vectors $v_i=h_i(v_0)$ form a basis of $\mathbb{R}^d$. For $\delta > 0$ sufficiently small depending on the $h_i$'s only, we can find $g_i \in X$ such that  
    $$\left|\langle g_i(v_0), \bigoplus_{j\neq i}\mathbb{R}g_j(v_0)\rangle\right| \leq  \frac{\left|\langle h_i(v_0), \bigoplus_{j\neq i}\mathbb{R}h_j(v_0) \rangle\right|+1}{2}=\alpha_i < 1$$
    for all $i,j$. Here, for a vector $v$ and a subspace $V$, $\langle v,V \rangle$ denotes $\sup_{w \in V, |w|=1} \langle v,w\rangle$. One can then deduce that 
    $$ \bigcap_{h \in X} h(C) \subset \bigcap_{i=1}^d g_i(C) \subset (\bigcap_{i=1}^d (g_i(v_0)^{\perp}) + \B(c',O_{H}(dl)) =\B(c',O_{H}(dl))$$
    where the implied constants depend on the $\alpha_i$'s only and $c' \in \mathbb{R}^d$.
\end{proof}

\subsection{The Prékopa--Leindler inequality and its stability}
The key inequality is a functional generalisation of the Brunn--Minkowski inequality called the \emph{Prékopa--Leindler inequality}.

\begin{proposition}[Prékopa--Leindler inequality]\label{Proposition: PL}
    Let $ \lambda \in [0;1]$ and $f,g,h: \mathbb{R}^n \rightarrow \mathbb{R}_{\geq 0}$ be three measurable functions. Suppose that 
    $$ \forall x,y \in \mathbb{R}^n, h((1-\lambda)x + \lambda y) \geq f(x)^{1-\lambda}g(y)^\lambda.$$
    Then 
    $$ \left\Vert h \right\Vert_1  \geq \left\Vert f\right\Vert_1^{1-\lambda}\left\Vert g \right\Vert_1^\lambda.$$
\end{proposition}

See \cite[\S 7]{BrunnMinkowski02Gardner} and references therein. Equality in the Prékopa--Leindler inequality is attained when $h$ is log-concave, see \cite{zbMATH03517415}. A function $f: \mathbb{R}^n \rightarrow \mathbb{R}_{\geq 0}$ is \emph{log-concave} if $D(f):=\{f > 0\}$ is a convex set and $f(x)=e^{-h(x)}$ for some convex function $h: C \rightarrow \mathbb{R}$.  It has been recently shown that the Prékopa--Leindler inequality satisfies a \emph{stability} result: 

\begin{theorem}[Quantitative stability of the Prékopa--Leindler inequality, \cite{zbMATH07977296, arXiv:2501.04656}]\label{Theorem:Stab PL}
    Let $ \lambda \in [0;1]$ and $\epsilon > 0$.  Let $f,g,h: \mathbb{R}^n \rightarrow \mathbb{R}_{\geq 0}$ be three measurable functions. Suppose that 
    $$ \forall x,y \in \mathbb{R}^n, h((1-\lambda)x + \lambda y) \geq f(x)^{1-\lambda}g(y)^\lambda.$$ Suppose moreover that 
    $$ \left\Vert h \right\Vert_1  \leq (1+\epsilon)\left\Vert f\right\Vert_1^{1-\lambda}\left\Vert g \right\Vert_1^\lambda.$$
    There are a log-concave function $\tilde{h} : \mathbb{R}^n \rightarrow \mathbb{R}_{\geq 0}$ and $w \in \mathbb{R}^n$ such that 
    $$ \left\Vert h - \tilde{h}\right\Vert_1 + \left\Vert a^\lambda f  - \tilde{h}( \cdot + \lambda w)\right\Vert_1 + \left\Vert a^{1-\lambda} g  - \tilde{h}( \cdot + (1-\lambda) w)\right\Vert_1 \ll_{n,\tau} \epsilon^{c} \left\Vert h \right\Vert_1$$
    where $a = \parallel f \parallel_1 / \parallel g \parallel_1$ and $c > 0$ is an absolute constant.
\end{theorem}

\subsection{Nonabelian Fourier theory and quasi-randomness}

We will also rely upon \emph{non-abelian Fourier theory} - we will follow the treatment from \cite[\S 3]{zbMATH07599280}. Given a compact simple Lie group $G$, the Peter--Weyl theorem asserts that 
$$ L^2(G) = \widehat{\bigoplus_{\pi \in \Sigma }}E_\pi^{\dim(\pi)}$$
where $\Sigma$ denotes the set of finite dimensional irreducible unitary representations of $G$ and $\widehat{\bigoplus}$ denotes the topological closure of the direct sum. From this, one deduces a decomposition of $L^2$ functions in series reminiscent of Fourier series. Given $\pi \in \Sigma$, the Fourier coefficient of a function $f$ is a $\dim \pi \times \dim \pi$ matrix defined as 
$$\hat{f}(\pi):= \int_G f(x) \overline{\pi(x)}dx.$$
Cornerstones of Fourier theory also hold in the non-abelian framework, such as Parseval's identity,
$$ \left\Vert f \right\Vert_2^2 = \sum_{\pi \in \Sigma} \dim \pi \left\Vert \hat{f}(\pi)\right\Vert_{HS}^2$$
where $\left\Vert A \right\Vert_{HS}=\left(\tr(AA^*)\right)^{1/2}$ denotes the Hilbert-Schmidt norm, the product formula $\widehat{f*g}(\pi) = \hat{f}(\pi)\hat{g}(\pi)$ or the inversion formula
$$f(x) = \sum_{\pi \in \Sigma} \dim \pi \tr \left(\hat{f}(\pi) \overline{\pi(x)^*}\right).$$

Because there are only finitely many representations of $G$ below a certain dimension, Gowers and Long observed in \cite{zbMATH07286399} that $G$ obeys a form of mixing property at small scales. More precisely, following their proof along with quantitative refinements from \cite{zbMATH07599280} we obtain:

\begin{lemma}\label{Lemma: Gowers-Long}
    Let $f,g \in L^2(G)$ and $\eta, D > 0$. Then 
  $$ \left\Vert f * g  - f * g * \chi_\eta \right\Vert_2^2 \ll_d  D\eta \left\Vert f * g \right\Vert_2^2 + D^{-1}\left\Vert f \right\Vert_2^2\left\Vert g \right\Vert_2^2.$$
\end{lemma}

\begin{proof}
    For any $f \in L^2(G)$ and $D \geq 0$, write 
    $$ L(f;D) := \sum_{\pi \in \Sigma, \dim \pi \leq D} \dim \pi \left\Vert \hat{f}(\pi)\right\Vert_{HS}^2$$ and  $$ H(f;D) := \sum_{\pi \in \Sigma, \dim \pi > D} \dim \pi \left\Vert \hat{f}(\pi)\right\Vert_{HS}^2.$$
    Note that the Parseval identity - see \cite[\S 3]{zbMATH07599280} - yields
    $$ \left\Vert f \right\Vert_2^2 = \sum_{\pi \in \Sigma} \dim \pi \left\Vert \hat{f}(\pi)\right\Vert_{HS}^2 = L(f;D) + H(f;D).$$
    Now, by \cite[Lem. 6.1]{zbMATH07599280} $$H(f*g;D) \leq D^{-1}\left\Vert f \right\Vert_2^2\left\Vert g \right\Vert_2^2$$ for every $f,g \in L^2(G)$. Moreover, $I -\hat{\chi}_\eta(\pi)$ has operator norm at most $O_d(\eta\dim \pi)$ with respect to the Hilbert--Schmidt norm, see \cite[Prop. 5.3 and Lem. 6.2]{zbMATH07599280}. So $$L(f * g - f * g *\chi_\eta; D) \ll_d  D\eta \left\Vert f * g \right\Vert_2^2$$
    by the product formula, see \cite[\S 3]{zbMATH07599280}.
    Finally, 
    $$ \left\Vert f * g - f * g *  \chi_\eta \right\Vert_2^2 \ll_d  D\eta \left\Vert f * g \right\Vert_2^2 + D^{-1}\left\Vert f \right\Vert_2^2\left\Vert g \right\Vert_2^2.$$
   
\end{proof}

This implies a form of \emph{quasi-randomness} in neighbourhoods of the identity: 

\begin{proposition}\label{Proposition: Quasi-random}
    Let $\eta > 0$ and $G$ be a compact simple Lie group of dimension $d$. Let $A,B$ be two subsets of $\B(e,\rho)$ with measure at least $\eta\mu(\B(e,\rho))$. Then for all $\epsilon > 0$ and $\rho \ll_{\eta,d, \epsilon} 1$ we have $c \gg_{C,d,\epsilon} 1$ such that for all $x$ in a subset of size $\gg_{\eta,d,\epsilon} \mu(\B(e,\rho))$,
    $$\mu(\B(x,\rho^{1+\epsilon}) \cap AB) \geq (1-\rho^c) \mu(\B(e,\rho^{1+\epsilon})).$$
    In particular, for $\rho \ll_{\eta,d, \epsilon} 1$, 
    $$\mu(\B(e,\rho^{1+\epsilon})) \subset ABB^{-1}A^{-1}.$$
\end{proposition}

Note that similar statements appear in \cite{zbMATH07286399, zbMATH07599280}.

\begin{proof}
    We have $\left\Vert \mathbf{1}_A\right\Vert_2^2, \left\Vert \mathbf{1}_B\right\Vert_2^2 \gg_{d} \eta\rho^{d}$. Moreover, $\rho^{3d} \ll_{\eta,d}\left\Vert \mathbf{1}_A * \mathbf{1}_B\right\Vert_2^2 \ll_{\eta,d} \rho^{3d}$, see e.g. \cite[P.9]{TaoProductSet08}. By Lemma \ref{Lemma: Gowers-Long}, for all $\epsilon, D > 0$ we have 
    \begin{align*}\left\Vert \mathbf{1}_A * \mathbf{1}_B - \mathbf{1}_A * \mathbf{1}_B *  \chi_{\rho^{1+2\epsilon}}\right\Vert_2^2 &\ll_d   D\rho^{1+2\epsilon} \left\Vert \mathbf{1}_A * \mathbf{1}_B \right\Vert_2^2 + D^{-1}\left\Vert \mathbf{1}_A\right\Vert_2^2\left\Vert \mathbf{1}_B \right\Vert_2^2 \\ &\ll_{\eta,d}D\rho^{1+2\epsilon} \left\Vert \mathbf{1}_A * \mathbf{1}_B \right\Vert_2^2 + D^{-1}\rho^{-1}\left\Vert \mathbf{1}_A * \mathbf{1}_B \right\Vert_2^2.\end{align*}
    So taking $D = \rho^{-1 - \epsilon}$ we have 
    $$\left\Vert \mathbf{1}_A * \mathbf{1}_B - \mathbf{1}_A * \mathbf{1}_B *  \chi_{\rho^{1+2\epsilon}} \right\Vert_2^2 \ll_{\eta,d} \rho^{\epsilon}\left\Vert \mathbf{1}_A * \mathbf{1}_B \right\Vert_2^2.$$
    But for every $g \in \B(e,\rho^{1+3\epsilon})$, Young's convolution inequality implies, 
    $$ \left\Vert \mathbf{1}_A * \mathbf{1}_B *  \chi_{\rho^{1+2\epsilon}} - \mathbf{1}_A * \mathbf{1}_B *  \chi_{\rho^{1+2\epsilon}}(g\ \boldsymbol{\cdot})\right\Vert_2 \leq \left\Vert\mathbf{1}_A * \mathbf{1}_B\right\Vert_2\left\Vert \chi_{\rho^{1+2\epsilon}} - \chi_{\rho^{1+2\epsilon}}(g \ \boldsymbol{\cdot})\right\Vert_1 \ll_d \rho^{\epsilon} \left\Vert\mathbf{1}_A * \mathbf{1}_B\right\Vert_2.$$
    Thus, 
    $$\left\Vert \mathbf{1}_A * \mathbf{1}_B - \mathbf{1}_A * \mathbf{1}_B(g\ \boldsymbol{\cdot})\right\Vert_2 \ll_{\eta,d} \rho^{\epsilon/2} \left\Vert\mathbf{1}_A * \mathbf{1}_B\right\Vert_2.$$
    In particular, 
    $$\frac{1}{\mu(\B(e,\rho^{1+\epsilon}))}\int_{\B(e, \rho^{1+3\epsilon})}\int_G \left| \mathbf{1}_A * \mathbf{1}_B(x) - \mathbf{1}_A * \mathbf{1}_B(gx)\right|^2dx dg \ll_{\eta,d} \rho^\epsilon\left\Vert\mathbf{1}_A * \mathbf{1}_B \right\Vert_2^2.$$
    By Fubini’s theorem, Markov’s inequality and the $L^2$ estimate on $\mathbf{1}_A * \mathbf{1}_B$, we conclude that there is $c \gg_{\eta,d,\epsilon}1$ such that for $\rho \ll_{\eta,d,\epsilon} 1$ and for all $x \in G$ in a subset of measure $\gg_{\eta,d,\epsilon} \mu(\B(e,\rho))$ we have $$\mu\left(\B(x,\rho^{1+3\epsilon}) \cap supp( \mathbf{1}_A * \mathbf{1}_B)\right) \geq (1- \rho^c)\mu(\B(x,\rho^{1+3\epsilon})).$$
    Since $supp( \mathbf{1}_A * \mathbf{1}_B) \subset AB$, the proof of the first part is complete. The second part follows immediately.
\end{proof}

We invite the interested reader to read \cite{zbMATH07286399} and \cite{zbMATH07599280} for more on this theme. Note that the same statement fails in a non-semi-simple group. For instance, the family of subsets $\left(\bigcup_{i=0}^{n}[\frac{i}{n}; \frac{i}{n} + \frac{1}{10n}]\right)_{n \geq 0}$ provides a counterexample in $\mathbb{R}$. Finally, even though a control as precise as the one obtained in Proposition \ref{Proposition: Quasi-random} certainly helps implying the proof of Theorem \ref{Theorem: Local stability}, our approach is robust and would work with much weaker bounds.

\subsection{Combinatorial lemma}

We conclude this section with a combinatorial lemma exploiting more finely the form of Brunn--Minkowski inequality. This will be useful both in the local and the global context.

\begin{lemma}\label{Lemma: Locality}
    Let $A,B \subset G$ and $d' > 0$. Suppose that $\tau^{-1}\mu(B) \leq \mu(A) \leq \tau \mu(B)$ and $$\mu(AB) \leq (1+ \epsilon)\left(\mu(A)^{1/d'} + \mu(B)^{1/d'} \right)^{d'}.$$ Suppose moreover that for  every $X \subset A $ and $Y \subset B$ we have 
    \begin{equation}
        \mu(XY) \geq (1-\epsilon)\left(\mu(X)^{1/d'} + \mu(Y)^{1/d'} \right)^{d'}. \label{Eq: local BM combinatorial lemma}
    \end{equation}

    If now  $C \subset A$ is such that $\mu(C) \geq \eta \mu(A)$, then there are $F_1, F_2 \subset G$ with $|F_1|=|F_2|=O_{d',\tau}(\eta^{-1})$ and $m = O_{d',\tau,\eta}(1)$ such that:
    \begin{enumerate}
        \item $A \subset F_1(C^{-1}C)^{m}$ and $B \subset (C^{-1}C)^{m}F_2$;
        \item The families $(f_1(C^{-1}C)^{100m})_{f_1 \in F_1}$ and $((C^{-1}C)^{100m}f_2)_{f_2 \in F_2}$ are both made of pairwise disjoint sets;
        \item For all $f_1 \in F_1, f_2 \in F_2$ we have the estimate
        $$\mu(f_1(C^{-1}C)^{O_\eta(1)} \cap A) = (1+O_{d',\tau, \eta}(\epsilon))\frac{\mu(A)}{|F_1|}  $$
        and 
        $$ \mu((C^{-1}C)^{O_\eta(1)}f_2 \cap B) = (1+O_{d',\tau, \eta}(\epsilon))\frac{\mu(B)}{|F_2|}.$$
    \end{enumerate}
\end{lemma}

\begin{proof}
Let $F \subset B$ be a subset such that $(Cf)_{f\in F}$ is a family of pairwise disjoint sets and choose it to be maximal for inclusion among such subsets. Then 
$$\mu(C)|F| \leq \mu(CF) \leq \mu(AB) \leq (1+\epsilon)(1+\tau^{1/d'})^{d'}\mu(A).$$
So $|F|\leq \eta^{-1}(1+\epsilon)(1+\tau^{1/d'})^{d'}$. By Rusza's covering lemma (e.g. \cite[Lem. 3.6]{TaoProductSet08}) $B \subset C^{-1}CF$. In addition, by Ruzsa's triangle inequality \cite[Lem. 3.2]{TaoProductSet08}, we have $\mu(AA^{-1}) \ll_{d',\tau}\mu(A)$. So by a similar argument, $A \subset F'(C^{-1}C)$ for some $F' \subset A$ of size $O_{d',\tau}(\eta^{-1})$.

Now, there are $m \leq 100^{|F'|+|F|}$ and subsets $\tilde{F'}\subset F'$ and $\tilde{F} \subset F$ such that $A \subset \tilde{F}'(C^{-1}C)^{m}$ and $B \subset (C^{-1}C)^{m}\tilde{F}$ and the family $(f(C^{-1}C)^{100m})_{f \in \tilde{F}'}$ (resp. $((C^{-1}C)^{100m}f)_{f \in \tilde{F}}$) is made of pairwise disjoint sets, see \cite[Lem. 4.3]{Machado2024MinDoubling} for detail. For notational simplicity, write again $\tilde{F}$ as $F$, $\tilde{F}'$ as $F'$ and $(C^{-1}C)^{m}$ as $C'$.

Let $f_1 \in F', f_2 \in F$ be such that $\mu(f_1C' \cap A) =\max_{f \in F'} \mu(fC' \cap A)$ and $\mu(C'f_2 \cap B) =\max_{f \in F} \mu(C'f \cap B)$. Suppose as we may that $\frac{\mu(f_1C' \cap A)}{\mu(A)} \leq \frac{\mu(C'f_2 \cap B)}{\mu(B)}$.

We know that the subsets $\left((fC' \cap A)(C'f_2 \cap B)\right)_{f\in F'}$ are pairwise disjoint. So 

$$ \mu(AB) \geq \sum_{f\in F'} \mu\left((fC' \cap A)(C'f_2 \cap B)\right).
$$
By (\ref{Eq: local BM combinatorial lemma}) now, 
\begin{equation}
    \mu(AB)  \geq \sum_{f\in F'} (1-\epsilon)\left(\mu(fC' \cap A)^{1/d'}+\mu(C'f_2 \cap B)^{1/d'}\right)^{d'}. \label{Eq: Local sums combinatorial lemma}
\end{equation}

We will first show that all subset $fC' \cap A$ are reasonably large. Let $\delta > 0$ to be chosen later and set $F'_{\geq}$ as the set of those $f \in F'$ such that $\mu(fC' \cap A) \geq \delta\mu(A)$ and write $F'_{<}=F'\setminus F'_{>}$. Now, 
$$\mu(AB)  \geq \sum_{f\in F'_{>}} (1-\epsilon)\mu(fC' \cap A)\left(1 + \frac{\mu(C'f_2 \cap B)^{1/d'}}{\mu(fC' \cap A)^{1/d'}}\right)^{d'} +  \sum_{f\in F'_{<}} (1-\epsilon)\mu(C'f_2 \cap B).$$
Since $\frac{\mu(f_1C' \cap A)}{\mu(A)} \leq \frac{\mu(C'f_2 \cap B)}{\mu(B)}$ and $\mu\left(C'f_2 \cap B\right) \geq \frac{\mu(B)}{|F|}$, we find 
\begin{align*}
     \mu(AB) &\geq \sum_{f\in F'_{>}} (1-\epsilon)\mu(fC' \cap A)\left(1 + \frac{\mu(B)^{1/d'}}{\mu( A)^{1/d'}}\right)^{d'} + \mathbf{1}_{F'_{<}\neq \emptyset}\frac{\mu(B)}{|F|} \\
     & \geq (1-\epsilon)(1-|F'|\delta)\mu(A)\left(1+\frac{\mu(B)^{1/d'}}{\mu( A)^{1/d'}}\right)^{d'}  +  \mathbf{1}_{F'_{<}\neq \emptyset}\frac{\mu(B)}{|F|}\\
     &= (1-\epsilon)(1-|F'|\delta)\left(\mu(A)^{1/d'}+\mu(B)^{1/d'}\right)^{d'}  +  \mathbf{1}_{F'_{<}\neq \emptyset}\frac{\mu(B)}{|F|}
\end{align*}
But $\mu(AB) \leq (1+\epsilon) \left(\mu(A)^{1/d'}+\mu(B)^{1/d'}\right)^{d'}$ by assumption. So
$$ \mathbf{1}_{F'_{<}}\frac{\mu(B)}{|F|} \ll_{d',\tau, \eta} \left(\epsilon + |F'|\delta\right)\left(\mu(A)^{1/d'}+\mu(B)^{1/d'}\right)^{d'}.$$
Since $|F'| \ll_{d', \tau, \eta}1$ and $\mu(A) \leq \tau \mu(B)$ we find
$$ |F'_{<}| \ll_{d',\tau, \eta} \epsilon + \delta.$$

Hence, if $\epsilon, \delta \ll_{d', \tau, \eta} 1$, the subset $F'_{<}$ must be empty.  Assume from now on that  $1 \ll_{d',\tau,\eta} \delta \ll_{d',\tau,\eta} 1$ is fixed such that for all $\epsilon$ sufficiently small depending on $\eta, \tau$ and $d'$ alone $F'_{<}$ is empty.

Recall that for all $f \in F'$, $\left(1 + \frac{\mu(C'f_2 \cap B)^{1/d'}}{\mu(fC' \cap A)^{1/d'}}\right)^{d'} \geq \left(1 + \frac{\mu( B)^{1/d'}}{\mu( A)^{1/d'}}\right)^{d'}$ and $\mu(AB) \leq (1+\epsilon) \left(\mu(A)^{1/d'}+\mu(B)^{1/d'}\right)^{d'}$. Combined with (\ref{Eq: Local sums combinatorial lemma}) we find  

$$\mu(fC' \cap A)\left|\left(1 + \frac{\mu( B)^{1/d'}}{\mu( A)^{1/d'}}\right)^{d'}-\left(1 + \frac{\mu(C'f_2 \cap B)^{1/d'}}{\mu(fC' \cap A)^{1/d'}}\right)^{d'}\right| \leq 2\epsilon \left(\mu(A)^{1/d'} + \mu(B)^{1/d'}\right)^{d'}.$$
But $\mu(fC' \cap A) \gg_{d',\tau,\eta} \mu(A)$ by the above discussion. So
$$\left|\left(1 + \frac{\mu( B)^{1/d'}}{\mu( A)^{1/d'}}\right)^{d'}-\left(1 + \frac{\mu(C'f_2 \cap B)^{1/d'}}{\mu(fC' \cap A)^{1/d'}}\right)^{d'}\right| \ll_{d',\tau,\eta} \epsilon.$$
Hence, by the mean value theorem, 
$$ \left|\frac{\mu(B)}{\mu(A)} -  \frac{\mu(C'f_2 \cap B)}{\mu(fC' \cap A)} \right|\ll_{d',\tau,\eta} \epsilon.$$
which yields $(1-O_{d',\tau, \eta}(\epsilon))\frac{\mu(C'f_2 \cap B)}{\mu(B)}= \frac{\mu(fC' \cap A)}{\mu(A)}$ for all $f \in F'$.

Now, picking any $f \in F'$ and repeating a symmetric argument, we find that \begin{equation}\frac{\mu(C'f \cap B)}{\mu(B)} = (1+O_{d',\tau, \eta}(\epsilon))\frac{\mu(f_3C' \cap A)}{\mu(A)}\label{Eq: rep. constante}\end{equation} for all $f_3 \in F$.

So we have (1),(2) and (3) for the choice $F_1=F'$ and $F_2=F$. It remains to show that $|F_1|=|F_2|$. But (\ref{Eq: rep. constante})  combined with $1=\sum_{f\in F_2} \frac{\mu(C'f \cap B)}{\mu(B)}$ yields 
$$ \frac{\mu(C'f \cap B)}{\mu(B)} =(1+O_{d',\tau, \eta}(\epsilon)) |F_2|$$
for all $f \in F_2$. Similarly, 
$$ \frac{\mu(fC' \cap A)}{\mu(A)} = (1+O_{d',\tau, \eta}(\epsilon)) |F_1|$$
for all $f \in F_1$. Now, because of (\ref{Eq: rep. constante}) again, 
$$(1-O_{d',\tau, \eta}(\epsilon))|F_1| \leq |F_2| \leq (1+O_{d',\tau, \eta}(\epsilon))|F_1|.$$
Because $|F_1|,|F_2| \ll_{\eta,d,\tau'}1$, we must have $|F_1|=|F_2|$ as soon as $\epsilon > 0$ is sufficiently small depending on $\eta, \tau$ and $d'$ alone.
\end{proof}
 
In particular, we have the corollary: 

\begin{corollary}\label{Corollary: Locality 99 percent}
    With $A,B, d'$ and $\tau, \epsilon$ as in Lemma \ref{Lemma: Locality}. Suppose moreover that $\mu(A \cap C) \geq \frac23 \mu(A)$. Then there are $g,g' \in G$ such that $ A \subset g(C^{-1}C)^{O_{d'}(1)}$ and $B \subset (C^{-1}C)^{O_{d'}(1)}g'$ as soon as $\epsilon > 0$ is sufficiently small depending on $d'$ alone. 
\end{corollary}

\section{Local stability}\label{Local stability}
Throughout this section, the results depend on a parameter $\rho > 0$ which will typically be small and parameter $\tau > 0$ which is fixed.

\subsection{Density functions and convex sets}

Fix $\rho > 0$ and $A, B \subset \B(e,\rho)$. To study the product $AB$, one is tempted by the following `principle': in a sufficiently small ball, a Lie group resembles its Lie algebra, and the product is close to the sum. In practice, this can be achieved by looking at \emph{density functions} (see \cite{jing2023measure}): given a scale $\eta > 0$ and a compact subset $X \subset G$ consider $\chi_\eta:=\frac{\mathbf{1}_{\B(e,\eta)}}{\mu(\B(e,\eta)}$ and set \begin{equation}
    \chi_{X,\eta} := \mathbf{1}_X * \chi_\eta = \frac{\mu\left(X \cap \B(\cdot, \eta)\right)}{\mu(\B(e,\eta))}. \label{Eq: Density functions}
\end{equation}
While one loses some information about $X$ by looking at density functions, the measure of $X$ is accessible thanks to the following simple formula:

\begin{equation}
\parallel\chi_{X,\eta}\parallel_1 = \mu(X). \label{Eq: L1 norm density}
\end{equation}

Density functions are smoother than indicator functions in many ways. For instance, they are continuous and their variations are well controlled below the chosen scale $\eta$:

\begin{fact}\label{Fact: Young + isoperimetry}
   Let $p \in [1;\infty]$ and $f \in L^p(G)$, $$\left\Vert f * \chi_\eta - f * \chi_\eta ( g\ \cdot )\right\Vert_p \ll_d \left\Vert f \right\Vert_p d(e,g) \eta^{-1}.$$ 
\end{fact}

\begin{proof}
    It is an easy consequence of Young's convolution inequality and a perimeter estimate for balls in $G$. 
\end{proof}

 When $\rho > 0$ is sufficiently small, $\B(e,\rho)$ can be identified with a neighbourhood of $0$ in the Lie algebra. We will use the additive notation $x+y$ for two elements $x,y \in \B(e,\rho)$ to mean the sum of the corresponding elements in the Lie algebra. Similarly, for $A \subset \B(e,\rho)$ and $\lambda>0$ we write $\lambda A$ to mean the result of rescaling $A$ in the Lie algebra by a factor $\lambda$. By Fact \ref{Fact: Lebesgue vs Haar}, 
 \begin{equation}
      \mu(\lambda A) = (1+O_d(\rho^2))\lambda^d\mu(A). \label{Eq: Rescaling}
 \end{equation}

In fact, this resemblance with Euclidean spaces can be made more detailed: 

\begin{fact}
    Let $A \subset \B(e,\rho)$ be a compact subset. Then: 

    \begin{enumerate}
        \item For any $\lambda \in \mathbb{R}_{\geq 0}$ such that $\lambda A \subset \B(e,\rho)$, $\mu(\lambda A) = (1+O_d(\rho^2))\lambda^d\mu(A)$;
        \item For any $g \in \B(e,\eta)$, $gA \subset Ag\B(e,O_d(\rho\eta))$ and  $Ag \subset gA\B(e,O_d(\rho\eta))$.
    \end{enumerate}

\end{fact}

 This property will also become relevant when considering \emph{convex subsets} in a neighbourhood of the identity. By a slight abuse of language, we call a subset $C$ of $\B(e,\rho)$ \emph{convex} if it is the image through the exponential map of a convex set in the Lie algebra. For any $A \subset \B(e,\rho)$ its \emph{convex hull} $\co(A)$ is the smallest convex set containing it. It is well defined as soon as $\rho$ is sufficiently small. Some elementary properties of convex subsets will be used repeatedly in this paper. 

 \begin{fact}
     Let $C \subset \B(e,\rho)$ be a convex subset such that $\mu(C) \geq \eta \mu(\B(e,\rho))$. Then there is $c \gg 1$ such that:
     \begin{enumerate}
         \item for all $g \in \B(e,\rho)$, $\mu(gC \Delta Cg) \ll_{d,\eta} \rho^c\mu(C)$;
         \item if $C'$ is any other convex subset, $\mu(CC' \Delta (C+C')) \ll_{d,\eta} \rho^c\mu(C)$;
         \item there is a convex subset $C' \subset \B(e,2\rho)$ with center of mass $0$ and $g \in \B(e,\rho)$ such that 
         $$\mu(C\Delta gC')\leq \rho^c\mu(C).$$
     \end{enumerate}
 \end{fact}

 \begin{proof}
     The proof follows from Lemma \ref{Lemma: Volume of thick convex sets} and Corollary \ref{Corollary: Sum vs product}.
 \end{proof}

 Similarly, we will say that a function $h: \B(e,\rho) \rightarrow \mathbb{R}$ is log-concave if $h \circ \exp$ is a log-concave function of the Lie algebra. Log-concave functions will arise in relation with the Prékopa--Leindler inequality. To extract information about convexity from log-concave functions, we will use Markov's inequality in the following formulation: 
\begin{lemma}\label{Lemma: Markov level set form}
    Let $f,g \in L^1(G)$. Then for all $t\geq \epsilon \geq 0$, 

    $$ \mu(\{|f - g| \geq \epsilon\}) \leq \epsilon^{-1}\parallel f- g \parallel_1$$
\end{lemma}

The main immediate consequence of this is a reformulation of our goal:

\begin{lemma}\label{Lemma: Symmetric difference and L^1-diff with densities}
    Let $A \subset \B(e,\rho)$ be a compact subset. For $0 < \eta \leq \rho^{101/100}$ and $C \subset \B(e,\rho)$ convex with $\mu(C) \geq \delta_1 \mu(\B(e,\rho))$, we have: 
    \begin{enumerate}
    \item If $\mu(A \Delta C) \leq \delta_2 \mu(C)$, then  $\parallel \chi_{A,\eta} - \mathbf{1}_C \parallel_1 \ll_{d,\delta_1} (\delta_2 + \rho)^c \mu(C)$ for some $c \gg_d 1$. 
    \item If $\parallel \chi_{A,\eta} - \mathbf{1}_C \parallel_1 \leq \delta \mu(C)$, then $\mu(A \Delta C) \ll_{d,\delta_1} (\delta_2 + \rho)^c  \mu(C)$ for some $c\gg_d 1$. 
    \end{enumerate}
\end{lemma}

\begin{proof}
We notice first that by a perimeter estimate, 
$$ \parallel \mathbf{1}_C - \mathbf{1}_C * \chi_{\eta}\parallel_1 \ll_{d,\delta_1} \rho^c$$
for some $c \gg_d 1$. 
    Let us now prove $(1) \Rightarrow (2)$. We have that 
    $$ \mu(A \Delta C) = \parallel \mathbf{1}_A - \mathbf{1}_C \parallel_1 \geq \parallel \mathbf{1}_A * \chi_\eta - \mathbf{1}_C * \chi_\eta\parallel_1 \geq (1 - O_{d,\delta_1}(\rho)^c)\parallel \chi_{A,\eta} - \mathbf{1}_C \parallel_1.$$ This concludes.

    Conversely, by Lemma \ref{Lemma: Markov level set form} applied to $f= \mathbf{1}_C$ and $g= \chi_{A,\eta}$ we have 
    $$ \mu(C \Delta \{\chi_{A,\eta} \geq 1 - \delta_2^{1/2}\}) \leq \delta_2^{1/2} \mu(C).$$
 
    So 
    $$ \mu(A \cap C\B(e,\eta)) \leq \int_C \chi_{A,\eta} \leq (1-\delta^{1/2})^2\mu(C).$$
    Since $\mu(C \Delta C\B(e,\delta_1)) \ll_{d,\eta} \rho^c$ this concludes.
    
\end{proof}

\subsection{A consequence of the stability of the Prékopa--Leindler inequality}
Our first result concerns local multiplicative inequalities for densities and product-sets.

\begin{lemma}\label{Lemma: mult. BM for densities}
    Let $1>\lambda > 0$ and $A,B \subset \B(e,\rho)$.  Then 
    $$\chi_{AB, \eta+O_d(\rho^2)}(x+y) \geq (1-O_d(\eta^{-1}\rho^2))\left(\chi_{A,(1-\lambda)\eta}(x)^{1/d} + \chi_{B,\lambda\eta}(y)^{1/d} \right)^d $$
    and 
    $$ \chi_{AB, \eta+O_d(\rho^2)}((1-\lambda)x + \lambda y ) \geq \left(1-O_d(\eta^{-1}\rho^2)\right)\chi_{(1-\lambda)^{-1}A,\eta}(x)^{(1-\lambda)}\chi_{\lambda^{-1}B,\eta}(y)^\lambda.$$
\end{lemma}

\begin{proof}
According to Corollary \ref{Corollary: Sum vs product} we have
$$ \B(xy, \eta) \subset \B(x+y,\eta + O_d(\rho^2)).$$
Therefore, 
 \begin{align*}
     AB \cap \B\left(x+ y, \eta + O_d(\rho^2)\right) &\supset  AB \cap \B\left(xy, (1-\lambda)\eta + \lambda \eta \right) \\
     & \supset \left(A \cap \B(x,(1-\lambda)\eta)\right)\left(B \cap \B(y,\lambda\eta)\right).
 \end{align*}
 By the local Brunn--Minkowski inequality \eqref{Eq: Local BM}, 
 $$ \mu\left(AB \cap \B(x+y, \eta + O_d(\rho^2))\right)\geq (1-O_d(\rho^2))\left(\mu(A \cap \B(x, (1-\lambda)\eta))^{1/d} +\mu(B \cap \B(y, \lambda\eta))^{1/d}\right)^{d}.$$
 This proves the first inequality.
 By (\ref{Eq: Rescaling}) and \eqref{Eq: Density functions}, 
 $$\frac{\mu(A \cap \B(x, (1-\lambda)\eta))}{\mu(\B(e,\eta))} \geq (1-O_d(\rho^2))(1-\lambda)^d\chi_{(1-\lambda)^{-1}A,\eta}((1-\lambda)^{-1}x),$$
  $$\frac{\mu(B \cap \B(x, \lambda\eta))}{\mu(\B(e,\eta))} \geq (1-O_d(\rho^2))\lambda^d\chi_{\lambda^{-1}B,\eta}(\lambda^{-1}x),$$
  and $$\mu\left(\B(e,\eta + O_d(\rho^2))\right) \leq (1 + O_d(\rho^2\eta^{-1}))\mu(\B(e, \eta)).$$
 
 So
 $$\chi_{AB, \eta+O_d(\rho^2)}(x+y) \geq (1-O_d(\rho^2\eta^{-1}))\left((1-\lambda)\chi_{(1-\lambda)^{-1}A,\eta}((1-\lambda)^{-1}x)^{1/d} + \lambda\chi_{\lambda^{-1}B,\eta}(\lambda^{-1}y)^{1/d} \right)^d. $$
  The second inequality then follows from the AM-GM inequality. 
\end{proof}

Lemma \ref{Lemma: mult. BM for densities} is precisely what we need to use the Prékopa--Leindler inequality (Proposition \ref{Proposition: PL}). 

\begin{proposition}\label{Proposition: Application PL stability}
    Let $A,B \subset \B(e,\rho)$ and $\epsilon > 0$ be such that \begin{equation}
        \mu(AB) \leq (1+\epsilon)\left(\mu(A)^{1/d} + \mu(B)^{1/d}\right)^d. \label{Eq: Stab condition 1}
    \end{equation}
    Let $\rho \geq \eta \geq \rho^{3/2}$. Then there is a log-concave function $h$ with support in $\B(e,\rho)$ such that 
    $$ \parallel h - \chi_{AB,\eta+O_d(\rho^2)}\parallel _1 + \parallel h_1 - \chi_{(1-\lambda_0)^{-1}A,\eta}\parallel _1 + \parallel h_2 - \chi_{\lambda_0^{-1}B,\eta}\parallel _1 \ll_{d}(\alpha + \rho)^c \parallel \chi_{AB,\eta+O_d(\rho^2)}\parallel _1$$
where $h_1, h_2$ are translated versions of $h$, $\lambda_0 := \frac{\mu(B)^{1/d}}{\mu(A)^{1/d}+\mu(B)^{1/d}}$ and $c \gg_d 1$.
\end{proposition}

\begin{proof}Note first that $\left\Vert \chi_{AB,\eta+O_d(\rho^2)} \right\Vert_1 = \mu(AB)$, $\left\Vert \chi_{(1-\lambda_0)^{-1}A,\eta} \right\Vert_1=\mu\left((1-\lambda_0)^{-1}A\right)$ and $\left\Vert \chi_{\lambda_0^{-1}B,\eta} \right\Vert_1= \mu(\lambda_0^{-1}B)$. According to \eqref{Eq: Rescaling},
$$\left\Vert \chi_{(1-\lambda_0)^{-1}A,\eta} \right\Vert_1 = (1-\lambda_0)^{-d}\mu(A) + O_{d}(\rho^2\mu(A)).$$
By our choice of $\lambda_0$ we get: 
$$\left\Vert \chi_{(1-\lambda_0)^{-1}A,\eta} \right\Vert_1 = \left(\mu(A)^{1/d} + \mu(B)^{1/d}\right)^{d} + O_{d}(\rho^2\mu(A)).$$
And by \eqref{Eq: Stab condition 1}, 
$$\left\Vert \chi_{(1-\lambda_0)^{-1}A,\eta} \right\Vert_1 = \mu(AB) + O_{d}(\epsilon + \rho^2)\mu(A) = (1+ O_d(\epsilon + \rho^2))\left\Vert \chi_{AB,\eta+O_d(\rho^2)} \right\Vert_1.$$
Similarly, 
$$\left\Vert \chi_{\lambda_0^{-1}B,\eta} \right\Vert_1 =  (1+ O_d(\epsilon + \rho^2))\left\Vert \chi_{AB,\eta+O_d(\rho^2)} \right\Vert_1.$$
As a consequence, we have the obvious inequality:
\begin{equation}
     \parallel \chi_{(1-\lambda_0)^{-1}A,\eta} \parallel_1^{1-\lambda_0}\parallel \chi_{\lambda_0^{-1} B, \eta}\parallel_1^{\lambda_0} \geq (1-O_{d,\tau}(\epsilon + \rho^{2}))\parallel\chi_{AB,\eta+O_d(\rho^2)}\parallel_1
\end{equation}

    Now, according to Lemma \ref{Lemma: mult. BM for densities} applied with $\lambda_0$ and $\eta$ as above, we have that $(1-O_d(\rho^{1/2}))^{-1}\chi_{AB,\eta+O_d(\rho^2)}$, $\chi_{(1-\lambda_0)^{-1}A,\eta}$ and $\chi_{\lambda_0^{-1}B,\eta}$ satisfy the condition of the Prékopa--Leindler inequality (Proposition \ref{Proposition: PL}). 

So we are in the stability case of the Prékopa--Leindler inequality. By Theorem \ref{Theorem:Stab PL}, there is a log-concave function $h$ defined on the Lie algebra and with support in $\B(e,\rho)$, such that 
$$ \parallel h - \chi_{AB,\eta+O_d(\rho^2)}\parallel _1 + \parallel h_1 - \chi_{(1-\lambda_0)^{-1}A,\eta}\parallel _1 + \parallel h_2 - \chi_{\lambda_0^{-1}B,\eta}\parallel _1 \ll_{d,\tau}(\epsilon + \rho)^c \parallel \chi_{AB,\eta+O_d(\rho^2)}\parallel _1$$
where $h_1, h_2$ are translated versions of $h$ (i.e. compositions of $h$ with a translation in the Lie algebra) and $c \gg_{d} 1$.
\end{proof}

\begin{corollary}\label{Corollary: Diff L1 norms}
    With the notations of Proposition \ref{Proposition: Application PL stability}. There are $x_1, x_2 \in \B(e,\rho)$ such that
    $$ \parallel \chi_{(1-\lambda_0)^{-1}A,\eta}(x_1 +\ \boldsymbol{\cdot}) - \chi_{AB,\eta+O_d(\rho^2)}\parallel _1  + \parallel \chi_{\lambda_0^{-1}B,\eta}(x_2 +\ \boldsymbol{\cdot})-\chi_{AB,\eta+O_d(\rho^2)}\parallel _1 \ll_{d,\tau}(\epsilon + \rho)^c \parallel \chi_{AB,\eta+O_d(\rho^2)}\parallel _1.$$
\end{corollary}

Finally, we obtain that subsets $A,B$ that are close to reaching equality in the Brunn-Minkowski inequality do so locally as well. This will be particularly useful later on.

\begin{corollary}[Pointwise stability]\label{Corollary: Stability condition at almost all points}
    With the notations of Proposition \ref{Proposition: Application PL stability}. Write $\nu$ the measure with density $\frac{h}{\parallel h \parallel_1}$ and let $x$ be picked at random following the law $\nu$. Then with probability $1-O_{d,\tau}(\alpha + \rho)^c$ where $c \gg_{d,\tau}1$ we have 
    $$\chi_{AB,\eta+O_d(\rho^2)}(x) \leq (1+O_{d,\tau}(\epsilon + \rho^2))\left(\chi_{A,(1-
    \lambda_0)\eta}((1-\lambda_0)(x-x_1))^{1/d} + \chi_{B,\lambda_0\eta}(\lambda_0(x-x_2))^{1/d}\right)^d.$$
\end{corollary}

In other words, Corollary \ref{Corollary: Stability condition at almost all points} asserts that if $A$ and $B$ satisfy the stability condition, then locally at scale $\eta$, they also satisfy the stability condition. This will be crucial in the next sections.

\begin{proof}
    Pick $\delta > 0$. By Proposition \ref{Proposition: Application PL stability}, we know that 
    $$ \left\Vert h - \chi_{AB,\eta+O_d(\rho^2)}\right\Vert_1 + \left\Vert h_1 - \chi_{(1-\lambda_0)^{-1}A,\eta}\right\Vert_1 + \left\Vert h_2 - \chi_{\lambda_0^{-1}B,\eta}\right\Vert_1 \ll_{d,\tau}(\epsilon + \rho)^c \left\Vert h\right\Vert_1$$
    for some $c \gg_d 1$. By Markov's inequality, we know that with probability $1-\delta^{-1}O_{d,\tau}(\epsilon + \rho)^c$ for $x$ picked $\nu$-randomly, $$\left|h(x) - \chi_{AB,\eta+O_d(\rho^2)}(x)\right|, \left|h(x+x_1) - \chi_{(1-\lambda_0)^{-1}A,\eta}(x)\right|, \left|h(x + x_2) - \chi_{\lambda_0^{-1}B,\eta}(x)\right| \leq \delta |h(x)|$$
    where $x_1, x_2 \in \B(e, \rho)$ are taken from the statement of Proposition \ref{Proposition: Application PL stability}. In particular,
    $$ (1+O_d(\delta))\left((1-\lambda_0) \chi_{(1-\lambda_0)^{-1}A,\eta}(x-x_1)^{1/d} + \lambda_0\chi_{\lambda_0^{-1}B,\eta}(x-x_2)^{1/d}\right)^d \geq \chi_{AB,\eta+O_d(\rho^2)}(x).$$

    Also,
    $$(1-\lambda_0) \chi_{(1-\lambda_0)^{-1}A,\eta}(x-x_1)^{1/d} \leq (1+O_d(\rho^2))\chi_{A,(1-\lambda_0)\eta}\left((1-\lambda_0)(x-x_1)\right)^{1/d}$$  and $$\lambda_0\chi_{\lambda_0^{-1}B,\eta}(x-x_2)^{1/d} \leq (1+O_d(\rho^2))\chi_{B,\lambda_0\eta}\left(\lambda_0(x-x_2)\right)^{1/d}$$
    by \eqref{Eq: Rescaling}. So 
    $$  \chi_{AB,\eta+O_d(\rho^2)}(x) \leq (1+O_d(\delta+\rho^2))\left(\chi_{A,(1-
    \lambda_0)\eta}((1-\lambda_0)(x-x_1))^{1/d} + \chi_{B,\lambda_0\eta}(\lambda_0(x-x_2))^{1/d}\right)^d.$$
    Taking $\delta = O_{d,\tau}(\epsilon + \rho)^{c/2}$ we get the corollary.
    
\end{proof}

To simplify notations, we will assume in the standing assumptions below that $x_1=x_2=0$. 

\begin{assumption}\label{Standing assumption} From now on, we will suppose that $A,B \subset \B(e,\rho)$ are two subsets such that 
\begin{equation}
        \mu(AB) \leq (1+\epsilon)\left(\mu(A)^{1/d} + \mu(B)^{1/d}\right)^d. \label{Eq: Stab condition}
    \end{equation}
We will moreover make two simplifications: 
\begin{enumerate}
    \item Upon replacing $A$ and $B$ with $gA$ and $Bh$ respectively, we will assume that the elements $x_1, x_2$ we obtain in Proposition \ref{Proposition: Application PL stability}, and Corollaries \ref{Corollary: Diff L1 norms} and \ref{Corollary: Stability condition at almost all points} are equal to $0$. 
    \item Upon considering $A' \subset A$ and $B' \subset B$ with $\mu(A\setminus A')=\mu(B\setminus B') =0$, we will assume that all points are density points. In particular, for all $a \in A'$ and all $\rho > 0$, $\mu(\B(e,\rho) \cap A) > 0$. 
\end{enumerate}
\end{assumption}

\subsection{Large convex subsets and convex hull}

In this section, we will build upon the existence of a log-concave function close to density functions to show that any convex subset that contains a fair proportion of $A$ has size proportional to that of the convex hull $\co(A)$ of $A$.

We turn to the first result of this section.

\begin{lemma}\label{Lemma: Convex control}
Let $A,B \subset \B(e,\rho)$ and $\epsilon > 0$ satisfy Assumption \ref{Standing assumption}.   Let $ \delta >0$. Let $C \subset \B(e,\rho)$ be a convex set such that $\mu(C \cap A) \geq \delta \mu(A)$. Then $$\mu(\co(A)) \ll_{d,\tau,\delta} \mu(C) + \rho^{d+\frac14}$$
    for $\rho, \epsilon$ sufficiently small depending on $d, \delta, \tau$ alone.
\end{lemma}

\begin{proof}
     Set $C_{0}:=C\B(e,\rho^{3/2})$. By Lemma \ref{Lemma: Locality}, $A \subset F_1 (C_0^{-1}C_0)^{m}$ with $|F_1|,m =O_{d,\eta,\tau}(1)$. In addition, the family $(f(C_0^{-1}C_0)^{100m})_{f \in F_1}$ is made of pairwise disjoint subsets and $\mu(A \cap f(C_0^{-1}C_0)^{m}) \geq \frac23\frac{\mu(A)}{|F_1|}$ for all $f \in F_1$ as soon as $\epsilon$ and $\rho$ are sufficiently small depending on $\tau, \eta$ and $d$ alone. 
    Let $h$ be the log-concave function obtained by applying Proposition \ref{Proposition: Application PL stability} to $A,B$ and $AB$ with $\eta=\rho^{3/2}$. Suppose that $\alpha, \rho>0$ are taken sufficiently small so that $\parallel \chi_{\lambda_0^{-1}A,\eta}-h_1\parallel _1 \leq \frac{1}{6|F_1|}\mu(A)$. Using log-concavity and the disjointness of $(f(C_0^{-1}C_0)^{100m})_{f \in F_1}$, we wish to show that $|F_1|=1$. 
    
    Take $f_1, \ldots, f_{d+1} \in F_1$  and $t_1, \ldots, t_{d+1} \in [0;1]$ such that $t_1+\ldots + t_{d+1}=1$. Write $C':= \sum_it_i\left(f_i(C_0^{-1}C_0)^{m}\right)$. Because $h$ is log-concave the Prékopa--Leindler inequality (Proposition \ref{Proposition: PL}) gives $$ \int_{(1-\lambda_0)^{-1}C'} h \geq \min_{i=1, \ldots, d+1} \int_{(1-\lambda_0)^{-1}\left(f_i(C_0^{-1}C_0)^{m}\right)}h \geq \frac{1}{2|F_1|}\mu(A).$$
    But $\parallel \chi_{(1-\lambda_0)^{-1}A,\eta}-h\parallel _1 \leq \frac{1}{6|F_1|}\mu(A)$. So 
    $$ \int_{(1-\lambda_0)^{-1}C'}  \chi_{\lambda_0^{-1}A,\eta} > 0$$
    which means $\left((1-\lambda_0)^{-1}A\right)_\eta \cap (1-\lambda_0)^{-1}C' \neq \emptyset$ and, hence, $C' \cap F_1 (C_0^{-1}C_0)^{m}\B(e,\eta) \neq \emptyset.$

    Moreover, according to Corollary \ref{Corollary: Sum vs product}, $C' \subset (\sum_it_if_i)(C_0^{-1}C_0)^m\B(e,\eta)$ for all $\rho$ sufficiently small depending on $d$ alone. Since $(C_0^{-1}C_0)^m\B(e,\eta) \subset (C_0^{-1}C_0)^{m+1}$, we have $(\sum_it_if_i) \in F_1(C_0^{-1}C_0)^{2m+2}$. Since the $f_i$'s and $t_i$'s are arbitrary, the Carathéodory theorem implies that the convex hull of $F_1$ is contained in $F_1(C_0^{-1}C_0)^{2m+2}$. But $(f(C_0^{-1}C_0)^{100m})_{f \in F_1}$ is made of disjoint open subsets, so $F_1$ must be made of no more than one element. 
    
 Write $F_1=\{f\}$, then $A \subset f(C_0^{-1}C_0)^{m} \subset f + \left(m(C_0-C_0)\right)\B(e,O_{d,\tau}(\rho^2))$ for $\rho$ sufficiently small, by successive applications of Corollary \ref{Corollary: Sum vs product}. By Corollary \ref{Corollary: Sum vs product} once more, $$\co(A) \subset f + \left(m(C_0-C_0)\right)\B(e,O_{d,\tau,\eta}(\rho^2))$$ for $\rho$ sufficiently small. And by Lemma \ref{Lemma: Volume of thick convex sets} we find $$\mu(\co(A)) \ll_{d} \mu\left(m(C_0-C_0)\right) + \rho^{d+\frac12} \ll_{d,\tau,\delta} \mu(C_0) + \rho^{d+\frac14} \ll_{d,\tau,\delta} \mu(C) + \rho^{d+\frac14} $$
 for $\rho$ sufficiently small depending on $\tau,\delta, d$ alone.  
\end{proof}

As a corollary we obtain a crucial result about the volume of neighbourhoods of $A,B$.

\begin{proposition}\label{Proposition:  Bounds max and support}
    Let $1 \leq c \leq \frac32$. Let $A,B \subset \B(e,\rho)$ and $\epsilon > 0$ satisfy Assumption \ref{Standing assumption}. Suppose that $\mu\left(\co(A)\right) > \delta \mu(\B(e,\rho)) > 0$. Let $ \eta=\rho^c $. Then:
    \begin{enumerate}
        \item $M \ll_{d,\tau} \frac{\mu(A)}{\mu(A_\eta)}$;
        \item If $h$ denotes the log-concave function provided, then there is $1<c'\ll_{d,\tau,\delta}1$ such that 
        $$C:=\{h \geq \frac{\mu(AB)}{c'\mu(\B(e, \rho))}\}$$ is a convex subset of measure $\gg_{d,\tau,\delta} \mu(\B(e,\rho))$ and $\mu((1-\lambda_0)C \setminus A_{\eta}) \ll_{ d,\tau,\delta} (\epsilon + \rho)^{c'}\mu(C)$ for some $c' \gg_{d,\tau} 1$. 
        \item In particular, $\mu(A_{\eta}) \gg_{d,\tau,\delta} \mu(\B(e,\rho)).$
    \end{enumerate}
\end{proposition}

\begin{proof}
     Let $h$ be the log-concave function obtained by applying Proposition \ref{Proposition: Application PL stability} to $A,B$ and $AB$ with $\eta$ chosen as in the statement. Note that the support of $\chi_{(1-\lambda_0)^{-1}A,\eta}$ is contained in $((1-\lambda_0)^{-1}A)_\eta$. Conversely, by Assumption \ref{Standing assumption} every point of $A$ is a density point, so $\mu\left(((1-\lambda_0)^{-1}A)_\eta\right) = \mu(\supp \chi_{(1-\lambda_0)^{-1}A,\eta})$. 
     
     Since $\chi_{(1-\lambda_0)^{-1}A,\eta}$ is continuous, it attains its maximum $M$ at some $(1-\lambda_0)^{-1}x_0 $. Also, $M \geq  \frac{\parallel \chi_{(1-\lambda_0)^{-1}A,\eta}\parallel_1}{\mu\left(((1-\lambda_0)^{-1}A)_\eta\right)}$.  By Lemma \ref{Lemma: mult. BM for densities} for $\rho > 0$ sufficiently small depending on $d$ alone, 
     \begin{equation}
         \chi_{AB,\eta+O_d(\rho^2)}(x_0 + \lambda_0y) \geq \frac{(1-\lambda_0)^d}{2}M \mathbf{1}_{(\lambda_0^{-1}B)_{\eta}}(y). \label{Eq: Using local BM and maximum}
     \end{equation} 
     Integrating, we find that 
     $$\mu(AB) \gg_{d,\tau} M \mu\left((\lambda_0^{-1}B)_{\eta}\right).$$
     So 
     $$ \frac{\mu\left((\lambda_0^{-1}B)_{\eta}\right)}{\mu(((1-\lambda_0)^{-1}A)_\eta)}\ll_{d,\tau}1.$$
     By a symmetric argument, 
     $$ \frac{\mu(((1-\lambda_0)^{-1}A)_\eta)}{\mu\left((\lambda_0^{-1}B)_{\eta}\right)}\ll_{d,\tau}1.$$
     By (\ref{Eq: Using local BM and maximum}) again for $C = \frac{(1-\lambda_0)^d}{2}$ we have 
     \begin{align*}
         \mu\left(\{x \in \B(e,\rho)| \chi_{AB,\eta+O_d(\rho^2)}(x) \geq CM\}\right) &\geq \mu\left(\lambda_0\left(\lambda_0^{-1}B\right)_{\eta}\right) \\
         &\gg_{d,\tau} \mu(((1-\lambda_0)^{-1}A)_\eta).
     \end{align*}
     So 
     $$\mu(AB) = \int \chi_{AB,\eta+O_d(\rho^2)} \gg_{d,\tau} M\mu(((1-\lambda_0)^{-1}A)_\eta).$$
     Since $\mu(AB) \ll_{d,\tau}\parallel \chi_{(1-\lambda_0)^{-1}A,\eta}\parallel_1$ we have in turn (1), i.e. 
     $$ M \ll_{d,\tau} \frac{\parallel \chi_{(1-\lambda_0)^{-1}A,\eta}\parallel_1}{\mu(((1-\lambda_0)^{-1}A)_\eta)}.$$
     We therefore have, 
     $$ \mu\left(\left\{x\in \B(e,\rho) | \chi_{(1-\lambda_0)^{-1}A,\eta}(x) \geq \frac{\parallel \chi_{(1-\lambda_0)^{-1}A,\eta}\parallel_1}{2\mu(((1-\lambda_0)^{-1}A)_\eta)}\right\}\right) \geq \frac{\parallel \chi_{(1-\lambda_0)^{-1}A,\eta} \parallel_1}{2M}  \gg_{d,\tau} \mu(((1-\lambda_0)^{-1}A)_\eta)$$
     where we have implicitly used $\mu\left(\supp \chi_{(1-\lambda_0)^{-1}A,\eta}\right) = \mu(((1-\lambda_0)^{-1}A)_\eta)$. But $\parallel \chi_{(1-\lambda_0)^{-1}A,\eta} - h_1\parallel_1  \leq O_{d,\tau}(\epsilon+\rho)^c \parallel \chi_{(1-\lambda_0)^{-1}A,\eta}\parallel_1 $ where $h$ is log-concave and $c \gg_{d,\tau} 1$. Write $$C:=\left\{x | h_1(x) \geq \frac{\parallel \chi_{(1-\lambda_0)^{-1}A,\eta}\parallel_1}{4\mu(((1-\lambda_0)^{-1}A)_\eta)}\right\}.$$ Then $C$ is convex by log-concavity of $h$. We have by Markov's inequality (Lemma \ref{Lemma: Markov level set form}), 

     \begin{align*} \mu\left(C \cap \left\{x\in \B(e,\rho) | \chi_{(1-\lambda_0)^{-1}A,\eta}(x) \geq \frac{\parallel \chi_{(1-\lambda_0)^{-1}A,\eta}\parallel_1}{2\mu(((1-\lambda_0)^{-1}A)_\eta)}\right\}\right) \  \   \   \   \   \  \  \ \  \  \  \  \  \  \  \  \  \  \  \  \ \; \; \;\;\;\;\;\;\;\;\;\;\;\;\;\;\;\;\;\;\;\;\\
     \geq \left(1 - O_{d,\tau}(\epsilon+\rho)^c\right)\mu\left( \left\{x\in \B(e,\rho) | \chi_{(1-\lambda_0)^{-1}A,\eta}(x) \geq \frac{\parallel \chi_{(1-\lambda_0)^{-1}A,\eta}\parallel_1}{2\mu(((1-\lambda_0)^{-1}A)_\eta)}\right\}\right)
     \end{align*}
     So for $\epsilon, \rho$ sufficiently small, we have 
     $$ \mu\left(C\right) \gg_{d,\tau} \mu(((1-\lambda_0)^{-1}A)_\eta)$$
      and 
     $$ \int_C \chi_{(1-\lambda_0)^{-1}A,\eta} \gg_{d,\tau} \parallel \chi_{(1-\lambda_0)^{-1}A,\eta}\parallel_1\gg_{d,\tau} \mu(A).$$
     Since $\left\Vert h \right\Vert_1 \leq \mu(AB) + O_{d,\tau}(\epsilon + \rho)^c \mu(A)$,  we also have 
     $$  \mu\left(C\right) \ll_{d,\tau} \mu((1-\lambda_0)^{-1}A)_\eta)\ll_{d,\tau} \mu(A_\eta)$$
    
     By Lemma \ref{Lemma: Convex control} applied to $C\B(e,\eta)$, for $\rho,\epsilon$ sufficiently small, 

     $$ \mu(\co(A)) \ll_{d,\tau} \mu(C\B(e,\eta)) + \rho^{d+\frac14}.$$
     By the assumption on $\mu(\co(A))$, 
     $$  \mu(\co(A)) \ll_{d, \tau} \mu(C\B(e,\eta))$$
     for $\rho$ sufficiently small. By Lemma \ref{Lemma: Volume of thick convex sets} and the assumption on $\mu(\co(A))$ again, 
      $$  \mu(\co(A)) \ll_{d,\tau} \mu(C).$$
      Finally, this implies, 
      $$ \mu(\co(A)) \ll_{ d,\tau} \mu(A_\eta).$$
      This concludes the proof of (2) and (3).
\end{proof}

Crucially, we will be able to expand the scope of Proposition \ref{Proposition:  Bounds max and support} past the scale $\rho^{3/2}$. 

\begin{proposition}\label{Proposition:  Bounds max and support, every scale}
    Let $1 \leq c $. Let $A,B \subset \B(e,\rho)$ and $\epsilon > 0$ satisfy Assumption \ref{Standing assumption}. Suppose that $\mu\left(\co(A)\right) > \delta \mu(\B(e,\rho) > 0$. Then $\mu(A_{\rho^c}) \gg_{d,\tau,c, \delta} \mu(\co(A))$.
\end{proposition}

\begin{proof}
    We will proceed by induction on $\lfloor \log_{\frac32} c\rfloor$. If $\lfloor \log_{\frac32} c\rfloor = 0$, then $c \leq \frac32$. So the conclusion already follows from Proposition \ref{Proposition:  Bounds max and support}. Suppose that the result is proved for all $c$'s with $m\geq \lfloor \log_{\frac32} c\rfloor$ and let $c_0 > 0$ be such that $m+1=\lfloor \log_{\frac32} c_0\rfloor$. Then $c_0 < (\frac32)^{m+2}$. In other words, we have $1<c_1:=c\frac32^{-m-1}<c_2= \frac32$. In addition, the result holds for $c_0c_1^{-1}$. 

    Now, (2) of Proposition \ref{Proposition:  Bounds max and support} applied to $A,B$ and $c_2$ provides a convex subset $C \subset \B(e,\rho)$ with measure $\mu(C) \gg_{d,\tau,\delta} \mu(\B(e,\rho))$ such that $\mu\left(C \setminus A_{\rho^{c_2}})\right) \ll_{d,\tau, \delta} (\epsilon + \rho)^{c'} \mu(\B(e,\rho))$ for some $c' \gg_{d,\tau} 1$.  For notational brevity, we will use $c'$ to denote different constants $\gg_{d,\tau} 1$ throughout the rest of this proof. For a proportion at least $(1-O_{d,\tau,\delta}(\epsilon + \rho)^{c'})$ of all $x \in C$, $\mu(B(x,\rho^{c_1}) \setminus A_{\rho^{c_2}}) \ll_{d,\tau, \delta} (\epsilon + \rho)^{c'}\mu(B(x,\rho^{c_1}))$. Moreover, by the definition of $C$, for all $x \in C$, $h(x) \gg_{d,\tau,\delta} \frac{\parallel h \parallel_1}{\supp h}$, see (2) of Proposition \ref{Proposition:  Bounds max and support}. By Corollary \ref{Corollary: Stability condition at almost all points}, for a proportion $(1-O_{ d,\tau,\delta}(\epsilon + \rho)^{c'})$ of all $x \in C$, the subsets $A_x:=A \cap \B(x,\rho^{c_1})$ and $B_x:=B \cap \B(\lambda_0 (1-\lambda_0)^{-1}x,\lambda_0 (1-\lambda_0)^{-1}\rho^{c_1})$ satisfy Assumption \ref{Standing assumption} with parameter $\epsilon'=O_{d,\tau}(\epsilon + \rho)^{c'}$, $\rho'=\rho^{c_1}$ and $\frac{\tau}{2}\leq\tau' \leq 2\tau $.

    By the above paragraph and the induction hypothesis applied to $A_x$ and $B_x$ at scale $\rho^{c_1}$, we have that for a proportion $(1-O_{ d,\tau,\delta,}(\epsilon + \rho)^{c'})$ of all $x \in C$, 
    $\mu( A\B(e, \rho^{c_1(c_0c_1^{-1})}) \cap \B(x,\rho^{c_1})) \gg_{d,\tau, \delta, d_0} \mu(\B(x,\rho^{c_1}))$. In particular, $$\mu(A\B(e, \rho^{c_1(c_0c_1^{-1})}) \cap C) \gg_{c_0, \delta, \tau, d} \mu(C) \gg_{ \delta, \tau, d} \mu(\B(e,\rho)).$$
\end{proof}
\subsection{Approximate subgroups with large covering number have large measure}

We will need one additional input to complete the proof of the local stability result. We have already established that up to any scale $\rho^c$ the density function has a rather large support i.e. $\mu(A_{\rho^c}) \gg_{d,\tau,\delta,c} \mu(\B(e,\rho))$. The results of this section will help us leverage this to show that $\mu(A) \gg_{d,\tau,\delta}  \mu(\B(e,\rho))$.

    \begin{proposition}\label{Proposition: Small doubling and large cov number implies large measure}
    Let $K, \tau \geq 0$. There is $1 < c_0 $ such that if $A,B \subset \B(e,\rho)$ are such that $\mu(AB) \leq K\mu(A)$, $\tau^{-1}\leq \frac{\mu(A)}{\mu(B)} \leq \tau$ and $\mu( A_{\rho^{c_0}}) \geq \delta \mu(\B(e,\rho))$, then $\mu(A) \gg_K \delta \mu(\B(e,\rho))$ as soon as $\rho$ is sufficiently small depending on $K$, $\tau$, $\delta$ and $d$ alone.
    \end{proposition}

To do so we will use a multi-scale analysis inspired from \cite{Machado2024MinDoubling}, and the notion of a $K$-approximate subgroup. A subset $\Lambda \subset G$ is a \emph{$K$-approximate subgroup} if $e \in \Lambda$, $\Lambda = \Lambda^{-1}$ and there is $F \subset G$ of size at most $K$ such that $\Lambda^2 \subset F\Lambda$. Two well-known facts will turn out to be particularly useful. First of all, if $A,B \subset G$ have positive Haar measure and satisfy $\mu(AB) \leq K\mu(A)$, $\tau\leq \frac{\mu(A)}{\mu(B)} \leq \tau^{-1}$ then there is a $(K\tau)^{O(1)}$-approximate subgroup $\Lambda \subset AA^{-1}$ such that $A$ is covered by $(K\tau)^{O(1)}$-translates of $\Lambda$. In particular, $\frac{1}{(K\tau)^{O(1)}}|A| \leq |\Lambda| \leq (K\tau)^{O(1)}|A|$, see \cite[Thm. 4.6]{TaoProductSet08}. Secondly, for every $m \geq 0$, $\Lambda^m \subset F^{m-1}\Lambda$ so 
$$ |\Lambda^m| \leq K^{O(m-1)}|A|.$$

    The proof of Proposition \ref{Proposition: Small doubling and large cov number implies large measure} therefore boils down to: 

    \begin{proposition}\label{Proposition: App. subgroup and large cov number implies large measure} Let $K \geq 0$ and $\Lambda \subset \B(e,\rho)$ be a $K$-approximate subgroup. There is $1 < c_0$ such that if $\mu(\Lambda_{\rho^{c_0}}) \geq \delta \mu(\B(e,\rho))$, then $\mu(\Lambda) \gg_{K,\tau} \delta \mu(\B(e,\rho))$ as soon as $\rho$ is sufficiently small depending on $K$, $\tau$, $\delta$ and $d$ alone.
    
\end{proposition}

We start with an intermediate lemma.

    \begin{lemma}\label{Lemma: Propagating containment of balls}
        There is $1<c_0 \ll_{d} 1$ such that for $\rho > 0$ sufficiently small, if $\B(e,\rho)\subset \Lambda_{\rho^{2c_0}}$ and $\lambda \in \Lambda$, then $$  \B\left(e,d(e,\lambda)\rho^{c_0}\right) \subset \left(\Lambda^{6d}\right)_{O_d(d(e,\lambda)\rho^{2c_0} + d(e,\lambda)^2\rho^{c_0})}.$$
    \end{lemma}

    \begin{proof}
        Note first that $[\lambda, \Lambda] \subset \Lambda^4$. Consider the map $g \mapsto [\lambda,g]$. The BCH formula (\ref{Eq: BCH formula}) implies in particular $[\exp(X),\exp(Y)]=\exp([X,Y])\exp(Z)$ where $|Z| \ll_{d} |X\parallel Y|\left(|X|+|Y|\right)$ - where we consider on $\mathfrak{g}$ the norm associated with the Riemannian metric fixed on $G$. Moreover, $\exp(X)\exp(Y) = \exp(X + Y) \exp(Z)$ with $|Z| \ll_d |X\parallel Y|$. For every $r > 0$, we thus have
        $$[\lambda, \B(g,r)] \subset [\lambda, g]\B\left(e,O_d(d(\lambda,e)r)\right).$$

        Therefore, 
        $$ [\lambda, \B(e,\rho)] \subset [\lambda, \Lambda \B(e,\rho^c)] \subset \Lambda^4 \B\left(e,O_d(d(e,\lambda)\rho^c)\right).$$
        Now,  $[\lambda, \B(e,\rho^{c'})]\B(e,d(e,\lambda)^2\rho^{c'})$ - and, hence, $\Lambda^4 \B\left(e,O_d(d(e,\lambda)\rho^c) + d(e,\lambda)^2\rho^{c'}\right)$  - contains a whole one-parameter sub-interval of the form $I:=\exp([0;Cd(\lambda,e)\rho^{c'}]X)$ for some $C\gg_d 1$ and $X \in \mathfrak{g}$ of unit norm. Since $\B(e,\rho)\subset \Lambda \B(e,\rho^{c})$, for each $i=1, \ldots, d$ there is $\lambda_i \in \Lambda$ such that the map 
        \begin{align*}
            [-\rho^C;\rho^C]^d &\longrightarrow G \\
            (t_1,\ldots,t_d) &\longmapsto \exp(t_1 \Ad(\lambda_1)X) \cdots \exp(t_d \Ad(\lambda_d)X)
        \end{align*}
        is $\rho^{-c_0}$-Lipschitz, see \cite[Lem. 2.16]{zbMATH06466329}. Finally, 
\begin{align*}
    \B\left(e,d(e,\lambda)\rho^{c'-c_0}\right) &\subset  \left(\prod \lambda_iI\lambda_i^{-1}\right) \\
    & \subset \Lambda^{6d}\B\left(e,O_d(d(e,\lambda)\rho^{c-c_0} + d(e,\lambda)^2\rho^{c'-c_0})\right).
\end{align*}
Taking $c'=c_0$ and $c=2c_0$ concludes.
    \end{proof}

We will use Lemma \ref{Lemma: Propagating containment of balls} in combination with a scale selection argument for approximate subgroups already used in \cite{Machado2024MinDoubling}.

\begin{lemma}[Boundedly many interesting scales, Lemma 4.4, \cite{Machado2024MinDoubling}]\label{Lemma: Boundedly many interesting scales}
    Let $\Lambda\subset G$ be a compact approximate subgroup and $m \geq 0$. Let $F \subset G$ finite be such that $\Lambda^m \subset  \Lambda F$ and write $|F|=n$. Define 
    $r_f:=\inf_{\lambda \in \Lambda} d(e,\lambda f)$ and let $r_1> \ldots > r_n$ be an enumeration of $\{r_f:f \in F\} \setminus \{0\}$. Then for all $i=1, \ldots, n$ (with $r_{n+1}=0$), 
    $$ \B(e,r_i) \cap \Lambda^m \subset \Lambda^2B_{r_{i+1}}.$$
\end{lemma}

    \begin{proof}[Proof of Proposition \ref{Proposition: App. subgroup and large cov number implies large measure}.]
    Let $\rho\geq r_1 > \ldots > r_n > 0, n \leq K^{24d+4}$ be the scales given by Lemma \ref{Lemma: Boundedly many interesting scales} applied to $\Lambda$ and $m=24d+4 $. Now, let $c_0$ be given by Lemma \ref{Lemma: Propagating containment of balls}. Let $r \geq 0$ be the least number such that for all $\rho > r' > r$ and $i$ such that $r_i \geq r' > r_{i+1}$ we have either (1), $B(e,r_i) \subset \Lambda^4B(e,r')$, or (2) $10r_{i+1} \geq r'$. Assume that $r \leq \rho^{3c_0}$. Suppose that $r>0$ and let $i$ be such that $r_i \geq r > r_{i+1}$. By construction, $r$ does not satisfy (2) i.e. $10r_{i+1} < r'$.  Moreover, there must be at least one $r'$ satisfying (1) in the range $10^{n} r\rho^{-c_0} \geq r' \geq r\rho^{-c_0}$. So there is $\lambda \in \Lambda^4$ such that ${10}^{n+1}2 r\rho^{-c_0} \geq d(\lambda,e) \geq 2r\rho^{-c_0}$ - provided $\rho$ is sufficiently small depending on $n$ alone. In particular, $d(e,\lambda) \leq \rho^{c_0}$ as soon as $\rho$ is sufficiently small in terms of $K$ and $d$ alone. By Lemma \ref{Lemma: Propagating containment of balls}, $$\B(e,d(e,\lambda)\rho^{c_0}) \subset \Lambda^{24d} \B(e,O_d(d(e,\lambda)\rho^{2c_0}))$$ and because $r \leq d(e,\lambda) \rho^{c_0}$, we have, 
    $$ \B(e,r_i) \subset \Lambda^4 \B(e,d(e,\lambda)\rho^{c_0}). $$

    Therefore,  $$\B(e,r_i) \subset \Lambda^{24d+4}\B(e,O_d(d(e,\lambda)\rho^{2c_0})).$$ For $\rho$ sufficiently small depending on $K$ and $d$ alone we have, with a slight abuse of notation, $O_d(d(e,\lambda)\rho^{2c_0}) \leq \frac{r}{10} \leq \frac{r_1}{2}$. Thus, by Lemma \ref{Lemma: Boundedly many interesting scales},
    $$ \B\left(e,\frac{r_i}{2}\right) \subset \left( \Lambda^{24d+4} \cap \B(e,r_i) \right)\B(e,O_d(d(e,\lambda)\rho^{2c_0}))\subset \Lambda^2 \B\left(e,O_d(d(e,\lambda)\rho^{2c_0}) + r_{i+1}\right) \subset \Lambda^2 \B\left(e,\frac{r}{5}\right). $$ 
    Thus, 
    $$ \B(e,r_i) \subset \Lambda^4 \B\left(e,\frac{2r}{5}\right).$$
   This implies that $r \leq \frac{2r}{5}$ i.e. $r=0$.

    To conclude, for all $i \leq n$ we have $\B(e,r_i) \subset \Lambda^4B(e,10r_{i+1})$. This implies by a straightforward induction that $\B(e,\rho)$ is covered by $100^{dn}$ translates of  $\Lambda^{4n}$. The result is established.
        
    \end{proof}

\begin{proof}[Proof of Proposition \ref{Proposition: Small doubling and large cov number implies large measure}.]
    By \cite[Thm. 4.6]{TaoProductSet08}, there is a $K^{O(1)}$-approximate subgroup $\Lambda \subset AA^{-1}$ such that $A$ is covered by $K^{O(1)}$ translates of $\Lambda$ and $\mu(\Lambda) \leq K^{O(1)}\mu(A)$. In particular, for all $c> 0$,  
   $$  \mu(A\B(e,\rho^c)) \leq K^{O(1)}\mu(\Lambda \B(e,\rho^c)).$$

    Let $c_0 > 0$ be given by Proposition \ref{Proposition: App. subgroup and large cov number implies large measure} and assume that $c > 2c_0$. By Proposition \ref{Proposition: Quasi-random} applied to $\Lambda \B(e,\rho^c)$,
    $$\B(e,\rho^{2}) \subset \left(\Lambda \B(e,\rho^c)\right)^4\subset \Lambda^4\B(e,4\rho^c)$$
    for $\rho$ sufficiently small depending on $K, d$ and $\delta$. Since $c > 2c_0$, Proposition \ref{Proposition: App. subgroup and large cov number implies large measure} applied to $\B(e, 2\rho^c) \cap \Lambda^8$ implies that $\B(e,\rho^c) \subset \Lambda^CF$ with $C = O_{K,d}(1)$ and $|F|= O_{K,d}(1)$. Hence,
    \begin{align*}
        \delta \mu(\B(e,\rho)) &\leq \mu(A\B(e,\rho^c)) \\ &\leq K^{O(1)} \mu(\Lambda \B(e,\rho^c)) \\&\leq K^{O(1)}\mu(\Lambda^{C+1}F) \\ &\ll_{K,d}\mu(\Lambda)  \ll_{K,d}\mu(A).
    \end{align*}
\end{proof}

    \subsection{Proof of local stability}

    \begin{proposition}\label{Proposition: Large enough measure for local stability}
    Let $A,B \subset \B(e,\rho)$ and $\epsilon > 0$ satisfying Assumption \ref{Standing assumption}. Suppose that $\mu(\co(A)) \geq \delta \mu(\B(e, \rho))$ for some $\delta > 0$. Then for all $\rho, \epsilon \ll_{d,\tau, \delta} 1$, we have $\mu(A) \gg_{d,\tau, \delta} \mu(\B(e, \rho))$.
\end{proposition}

\begin{proof}
    Let $c_0$ be as in Proposition \ref{Proposition: Small doubling and large cov number implies large measure}. By Proposition \ref{Proposition:  Bounds max and support, every scale},  
    $$\mu(A\B(e,\rho^{c_0}))  \gg_{d,\tau,\delta,c_0} \mu(\B(e,\rho)).$$ By Proposition \ref{Proposition: Small doubling and large cov number implies large measure}, we have $\mu(A) \gg_{d,\tau} \mu(\B(e,\rho))$ for $\epsilon, \rho$ sufficiently small depending on $d$ and $\tau$ alone. 
\end{proof}

\begin{proof}[Proof of Theorem \ref{Theorem: Local stability}.]
By Proposition \ref{Proposition: Application PL stability}, for every $\rho^{5/4} \geq \eta \geq \rho^{3/2}$ there is a log-concave function $h_\eta$ with support in $\B(e,\rho)$ such that 
    \begin{equation}
        \left\Vert h_\eta - \chi_{A,\eta} \right\Vert_1 \leq O_{d,\tau}(\epsilon + \rho)^c\mu(A) \label{Eq: L1 comp convex}
    \end{equation} 
    for some $c \gg_{d,\tau} 1$. Throughout the proof $c$ will denote various constants that all satisfy $c \gg_{d,\tau}1$. By Proposition \ref{Proposition: Large enough measure for local stability}, $\mu(A), \mu(B) \gg_{d,\tau, \delta} \mu(\B(e,\rho))$ as soon as $\epsilon$ and $\rho$ are sufficiently small. According to Proposition \ref{Proposition: Quasi-random}, we find that for all $x \in AB$ in a subset of measure $\gg_{d,\tau, \delta} \mu(\B(e,\rho))$, we have 
    $$\mu(\B(x,\eta) \cap AB) \geq (1- \rho^c)\mu(\B(x,\eta)).$$ In other words, 
    $\mu\left(\{\chi_{AB,\eta} \geq 1- \rho^c\}\right) \gg_{d,\tau,\delta} \mu(\B(e,\rho))$. By Markov's inequality and Corollary \ref{Corollary: Diff L1 norms} we have $$\mu\left(\{\chi_{A,(1-\lambda_0)\eta} \geq 1- \rho^c\}\right) \gg_{d,\tau, \delta} \mu(\B(e,\rho))$$ and 
    $$\mu\left(\{\chi_{B,\lambda_0\eta} \geq 1- \rho^c\}\right) \gg_{d,\tau,\delta} \mu(\B(e,\rho)).$$ 
    Now, if $x \in \{\chi_{A,(1-\lambda_0)\eta} \geq 1- \rho^c\}$ and $y \in \{\chi_{B,\lambda_0\eta} \geq 1- \rho^c\}$ we have 
    $$\chi_{AB,\eta/2}(xy)=1$$
    as soon as $\rho \ll_{d} 1$. Indeed, this is a consequence of:
    \begin{claim}\label{Claim je sais pas}
        Let $\lambda > 0$ and $\eta_1 > \eta_2 > 0$ with $\eta_1/\eta_2 \leq \lambda$. There is $\delta > 0$ such that if $X \subset \B(e,\eta_1)$ and $Y \subset \B(e,\eta_2)$ satisfy $\mu(X) \geq (1-\delta) \mu(\B(e,\eta_1))$  and $\mu(Y) \geq (1-\delta)\mu(\B(e,\eta_2))$, then 
        $\B(e,\frac{\eta_1+\eta_2}{2}) \subset XY$.

        In fact, if $C \subset \B(e,\eta_1)$ is a convex subset of measure at least $\delta' \mu(\B(e,\eta_1))$ and $X \subset (1-\lambda)C$, $Y \subset \lambda C$ have measure $(1-\delta)(1-\lambda)^dC$ and $(1-\delta)\lambda^d C$ respectively, then $XY$ contains a convex subset of $C$ of measure $(1-O_{\delta',\lambda}(\delta+\eta_1)^c)\mu(C)$ as soon as $\rho$ is sufficiently small depending on $\delta'$. 
    \end{claim}

    Write $C_\eta:=\{\chi_{AB,\eta} = 1\}$. We have shown in the previous paragraph that $\mu(C_{\frac{\eta}2}) \gg_{d,\tau, \delta} \mu(\B(e,\rho))$. Note moreover that 
    \begin{equation}
        \mu(C_\eta) \leq \int_{G}\chi_{AB,\eta} = \mu(AB). \label{Eq: Upper bound C}
    \end{equation}
    We will show that, conversely, 
    $$ \mu(C_{\eta}) \geq (1- O_{d,\tau,\delta}(\epsilon + \rho)^c)\mu(AB)$$
    for some $c \gg_{d,\tau} 1$. This will boil down to the combination of two claims. One regarding the (approximate) invariance of $\mu(C_\eta)$ when changing the scale $\eta$:

    \begin{claim}\label{Claim: Invariance between scales}
        Let $\rho^{5/4} \geq \eta_1 > \eta_2 \geq \rho^{3/2}$. Then 
        $$\mu(C_{\eta_1}) \leq \mu(C_{\eta_2}) \leq (1 + O_{d,\tau, \delta}(\epsilon + \rho)^c)\mu(C_{\eta_1})$$
        for some $c \gg_{d,\tau}1$. 
    \end{claim}

    The second one will be a Brunn--Minkowski-type growth between scales. 

    \begin{claim}\label{Claim: BM-type growth}
         Let $\rho^{5/4} \geq \eta \geq \rho^{3/2}$. Recall that $\lambda_0:= \frac{\mu(B)^{1/d}}{\mu(A)^{1/d} + \mu(B)^{1/d}}$. Then, 
         $$ \mu(C_{\frac{(1-\lambda_0)\eta}{2}}) \geq (1- O_{d,\tau,\delta}(\epsilon+\rho)^c)\left((1-\lambda_0)\mu(C_\eta)^{1/d} + \mu(B)^{1/d}\right)^{d}$$
         for some $c \gg_{d,\tau} 1$.
    \end{claim}

    For both claims to be true at once, we must have: 

    $$\mu(C_\eta) \geq (1-O_{d,\tau, \delta}(\epsilon + \rho)^c)\left((1-\lambda_0)\mu(C_\eta)^{1/d} + \mu(B)^{1/d}\right)^d.$$
    This yields, 
    $$\mu(C_{\eta})^{1/d} \geq (1-O_{d,\tau, \delta}(\epsilon + \rho)^c)\lambda_0^{-1}\mu(B)^{1/d}.$$
    In turn, this can be rewritten as,
    $$\mu(C_\eta) \geq (1-O_{d,\tau, \delta}(\epsilon + \rho)^c)\left(\mu(A)^{1/d} + \mu(B)^{1/d}\right)^d$$
    for some $c \gg_{d,\tau} 1$. Combining \eqref{Eq: Upper bound C} and Assumption \ref{Standing assumption} we find 
    \begin{equation}
        \mu(AB) \geq \mu(C_\eta) \geq (1 - O_{d,\tau, \delta}(\epsilon + \rho)^c) \mu(AB)
    \end{equation}
    for some $c \gg_{d,\tau}1$. In other words, 
    $$ \parallel \mathbf{1}_{C_\eta} - \chi_{AB,\eta}\parallel_1 \ll_{d,\tau, \delta} (\epsilon + \rho)^c \mu(AB).$$
    In turn, 
    $$ \parallel \mathbf{1}_{C_\eta} - h_\eta\parallel_1 \ll_{d,\tau, \delta} (\epsilon + \rho)^c \parallel h_\eta \parallel_1.$$
    So by Markov's inequality (Lemma \ref{Lemma: Markov level set form}), 
    $$\mu(C_\eta \Delta C_\eta') \ll_{d,\tau,\delta} (\epsilon + \rho)^{c/2}\mu(AB)$$
    where $C_\eta'$ is the convex level set $\{h_\eta \geq 1- O_{d,\tau,\delta}(\epsilon + \rho)^{c/2}\}$. According to Lemma \ref{Lemma: Symmetric difference and L^1-diff with densities}, this is enough to conclude, the fact that $C$ can be taken containing $A$ excepted. But this follows from the measure inequality 
    $$ \mu\left((A \setminus C)(B \cap \lambda C) \setminus (1+\lambda)C\right) \ll_{d,\tau} (\epsilon + \rho)^c \mu(C).$$

  Thus, it remains only to prove the two claims.

  \begin{proof}[Proof of Claim \ref{Claim: Invariance between scales}.]
   For notational simplicity, throughout this proof, $c$ will denote constants that might vary line by line, but all satisfy $c \gg_{d,\tau} 1$. 
   
   The left-hand side inequality is obvious. Indeed, if $\mu(\B(x,\eta_1) \cap AB) = \mu(\B(x,\eta_1))$, then $\mu(\B(x,\eta_2) \cap AB) = \mu(\B(x,\eta_2))$ since $\eta_1 \geq \eta_2$. Conversely, by Markov's inequality (Lemma \ref{Lemma: Markov level set form}), the convex level set $C_{\eta_2}':=\{h_{\eta_2} \geq 1- O_{d,\tau,\delta}(\epsilon + \rho)^{c}\}$ has measure $\geq (1- O_{d,\tau,\delta}(\epsilon + \rho)^c)\mu(C_{\eta_2})$. Now, by Corollary \ref{Corollary: Diff L1 norms}, 
   $$ \mu\left((1-\lambda_0)C_{\eta_2}' \cap \{\chi_{A,(1-\lambda_0)\eta_2}\geq 1 - O_{d,\tau,\delta}(\epsilon + \rho)^c) \}\right) \geq (1-O_{d,\tau, \delta}(\epsilon + \rho)^c) \mu((1-\lambda_0)C_{\eta_2}')$$
   and 
      $$ \mu\left(\lambda_0C_{\eta_2}' \cap \{\chi_{B,\lambda_0\eta_2}\geq 1 - O_{d,\tau,\delta}(\epsilon + \rho)^c) \}\right) \geq (1-O_{d,\tau, \delta}(\epsilon + \rho)^c) \mu(\lambda_0C_{\eta_2}').$$
      So for all $x \in (1-\lambda_0)C_{\eta_2}' \cap \{\chi_{A,(1-\lambda_0)\eta_2}\geq 1 - O_{d,\tau,\delta}(\epsilon + \rho)^c) \}$ and $y \in \lambda_0C_{\eta_2}' \cap \{\chi_{B,\lambda_0\eta_2}\geq 1 - O_{d,\tau,\delta}(\epsilon + \rho)^c) \}$ we have by Claim \ref{Claim je sais pas} that 
      $$ \chi_{AB,\eta_2/2}(xy) = 1.$$
      By the second part of Claim \ref{Claim je sais pas}, $C_{\eta_2/2}$ therefore contains a convex subset $C_{\eta_ 2/2}'$ of measure at least $(1-O_{d,\tau,\delta}(\epsilon + \rho)^c)\mu(C_{\eta_2})$. Since $\mu(C_{\eta_2/2}'') \gg_{d,\tau, \delta} \mu(\B(e,\rho))$, we know by a perimeter estimate that for all $x$ in a subset of measure at least $(1 - O_{d,\tau,\delta}(\rho)^c)\mu(C_{\eta_2/2}'')$ of $C_{\eta_2}''$ we have $\B(x,\eta_1) \subset C_{\eta_2}''$. Hence, $\chi_{AB,\eta_1}(x) = 1$ for any such $x$. This concludes the proof of the claim. 
       \end{proof}

       Finally, we prove the last claim.

       \begin{proof}[Proof of Claim \ref{Claim: BM-type growth}.]
           According to Markov's inequality, 
           $$ \{\chi_{A,(1-\lambda_0)\eta} \geq 1 - O_{d,\tau,\delta}(\epsilon + \rho^c)\}$$
           has measure at least $(1-O_{d,\tau,\delta}(\epsilon + \rho)^c)(1-\lambda_0)^d\mu(C_{\eta})$. For every $x \in \{\chi_{A,(1-\lambda_0)\eta} \geq 1 - O_{d,\tau,\delta}(\epsilon + \rho)^c\}$ and $y \in B$, we have $\chi_{AB,(1-\lambda_0)\eta}(xy) \geq 1 - O_{d,\tau,\delta}(\epsilon + \rho)^c$. So 
           $$ \mu\left(\{\chi_{AB,(1-\lambda_0)\eta} \geq 1 - O_{d,\tau,\delta}(\epsilon + \rho)^c\} \right) \geq (1 - O_{d,\tau,\delta}(\epsilon + \rho)^c) \left((1-\lambda_0)\mu(C_\eta)^{1/d} + \mu(B)^{1/d}\right)^d.$$
           We can now proceed as in the end of the proof of Claim \ref{Claim: Invariance between scales}. This yields,
           $$\mu(C_{(1-\lambda_0)\eta/2}) \geq  (1 - O_{d,\tau,\delta}(\epsilon + \rho)^c) \left((1-\lambda_0)\mu(C_\eta)^{1/d} + \mu(B)^{1/d}\right)^d.$$
       \end{proof}
    \end{proof}

\section{Global stability}\label{Global stability}
We will now conclude the proof of our global stability result (Theorem \ref{Theorem: Global stability}). Our strategy will roughly follow that of \cite{Machado2024MinDoubling} deducing global stability from local stability (Theorem \ref{Theorem: Local stability}). As in the proof of Theorem \ref{Theorem: Local stability}, we will mostly be able to maintain \emph{polynomial} errors and dependencies. The only part where this fails is in the invocation of \cite[Prop. 1.6]{Machado2024MinDoubling}, which comes with a double exponential bound.

 \subsection{Initial observation}\label{Subsection: Double counting}
    The proof of the global Brunn--Minkowski inequality in \cite[\S 5.5]{Machado2024MinDoubling} provides additional information. Let $\delta > 0$ and let $H$ be a proper closed subgroup. Recall that $H_\delta:=H\B(e,\delta)$.  Write $\mathfrak{h}$ the Lie subalgebra of $\mathfrak{g}$ associated with $H$ and $T_\delta:=\exp(\mathfrak{h}^\perp \cap \B_{\mathfrak{g}}(0,\delta))$. Then $H_\delta = H T_\delta$ and for all $X \subset H_\delta, h \in H$ and $\rho > 0$ write $$X_{h,\rho} = X \cap (H \cap \B(e,\rho))T_\delta.$$ 
    We have 
    \begin{equation}
        \int_H \mu(X_{h,\rho})dh = \mu(X)\mu_H(H \cap \B(e,\rho)) \label{Eq: Double counting}
    \end{equation}
    where $\mu_H$ denotes a Haar measure on $H$. We will denote $H \cap \B(e,\rho)$ by $\B_H(e,\rho)$. We will reduce the study of Minkowski products $XY$ to the study of Minkowski products $X_{h_1,\rho_1}Y_{h_2,\rho_2}$. The following is a useful formula in that regard 
    $$X_{h_1,\rho_1}Y_{h_2,\rho_2} \subset (XY)_{h_1h_2, \rho_1+\rho_2+ O_d(\delta^2)}, $$
    see \cite[\S 5.3]{Machado2024MinDoubling}. In this first subsection, we make the following assumptions:

    \begin{assumption}\label{Assumption semi-local}
        We have $\delta, \epsilon > 0$, a proper connected subgroup $H$ and $A,B \subset G$ such that $A,B,AB \subset H_\delta$ and 
        $$ \mu(AB) \leq (1+\epsilon)(\mu(A)^{1/d'}+\mu(B)^{1/d'})^{d'}$$
        where $d'$ denotes the codimension of $H$. Write 
        $$ \tau := \frac{\mu(A)}{\mu(B)} \text{ and } \lambda:=\tau^{1/d}.$$
        Define now 
    $$M:=\max_{h_1,h_2}\left\{\frac{\mu(A_{h_1, \rho})}{\mu(A)\mu_H(\B_H(e,\rho))}, \frac{\mu(B_{h_2,\lambda\rho})}{\mu(B)\mu_H(\B_H(e,\lambda\rho))}\right\},$$ and assume, as we may, that $M=\frac{\mu(B_{h,\lambda\rho})}{\mu(B)\mu_H(\B_H(e,\lambda\rho))}$ for some $h \in H$. Choose finally any $h_0 \in H$ such that $\frac{\mu(B_{h_0,\lambda\rho})}{\mu(B)\mu_H(\B_H(e,\lambda\rho))} \geq (1-\epsilon)M$. 
    \end{assumption}

    We found in \cite[\S 5.5]{Machado2024MinDoubling} the following series of inequalities under the assumption $\rho \geq \delta^{3/2}$:
    \begin{align}
        \mu(AB) &\geq \frac{\int_H \mu((AB)_{h,(1+\lambda)\rho + O_{d,\tau}(\delta^2)}) dh}{\mu_H(\B(e,(1+\lambda)\rho+ O_{d,\tau}(\delta^2)))} \label{Eq: global 1}\\
                &\geq \frac{\int_H \mu(A_{hh_0^{-1},\rho}B_{h_0,\lambda\rho})dh}{\mu_H(\B(e,(1+\lambda)\rho + O_{d,\tau}(\delta^2)))} \label{Eq: global 2}\\
                & \geq\frac{\int_H \mu(A_{hh_0^{-1},\rho})\left(1 + \frac{\mu(B_{h_0,\lambda\rho})^{1/d}}{\mu(A_{hh_0^{-1},\rho})^{1/d}}\right)^ddh }{ \mu_H(\B(e,(1+\lambda)\rho)+O_{d,\tau}(\delta^2)) } \label{Eq: global 3}\\
                & \geq (1-\epsilon)^d\left(1 + \frac{\mu_H(\B(e,\lambda\rho))^{1/d}\mu(B)^{1/d}}{\mu_H(\B(e,\rho))^{1/d}\mu(A)^{1/d}}\right)^d\frac{\int_H \mu(A_{hh_0^{-1},\rho})dh}{ \mu_H(\B(e,(1+\lambda)\rho+O_{d,\tau}(\delta^2))) } \label{Eq: global 4}\\
                & = (1-\epsilon)^d \frac{\mu_H(\B(e,\rho))}{\mu_H(\B(e,(1+\lambda)\rho+O_{d,\tau}(\delta^2)))} \left(1 + \frac{\mu_H(\B(e,\lambda\rho))^{1/d}\mu(B)^{1/d}}{\mu_H(\B(e,\rho))^{1/d}\mu(A)^{1/d}}\right)^d\mu(A) \label{Eq: global 5}\\
                & \geq (1-O_{d,\tau}(\epsilon+\rho + \delta)^c)\left(\mu(A)^{1/d'} + \mu(B)^{1/d'} \right)^{d'} \label{Eq: global 6}
    \end{align}
    where $c \gg_{d,\tau} 1$. Hence, if there is almost equality in the Brunn--Minkowski inequality, all the above inequalities are almost equalities. This will be the main observation that will enable us to reduce from global stability to local stability.





   \begin{proposition}\label{Proposition: Local mass is constant}
   Let $A, B \subset G$ satisfy Assumption \ref{Assumption semi-local} with parameters $\delta, \epsilon$.  Let $\rho \geq \delta^{3/2}$.
   Write $$m(\rho):= \max_{h \in H} \left( \mu(A_{h,\rho}) \right), \  H_{A,\rho}:=\{h \in H: A_{h,\rho}\neq \emptyset\}$$ and for $\alpha > 0$, $$H^\alpha_{A, \rho}:=\{h \in H: \mu(A_{h,\rho}) \geq (1-\alpha)m(\rho)\}.$$ Then  $\mu_H\left(H^{(\epsilon +\delta + \rho)^c}_{A, \rho}\right) \geq 1 - O_{d,\tau}(\epsilon +\delta + \rho)^c$ and $m(\rho) = (1 +O_{d,\tau}(\epsilon +\delta + \rho)^c)M$  for some $c \gg_{d,\tau} 1$. A similar statement holds for $B$.
   \end{proposition}

   From now on and throughout the rest of this subsection $c$ will denote possibly different positive constants that depend on $d$ and $\tau$ only. We will prove Proposition \ref{Proposition: Local mass is constant} in a few steps. First of all, notice that:

   \begin{fact}\label{Fact: Lower bound minimal}
   We have
       $$\sup_{h \in H} \mu(A_{h,\rho}) \geq \frac{\mu_H(\B(e,\rho))\mu(A)}{\mu(H_{A,\rho})} \text{ and } \sup_{h \in H} \mu(B_{h,\lambda\rho}) \geq \frac{\mu_H(\B(e,\lambda\rho))\mu(B)}{\mu(H_{B,\lambda\rho})}.$$
   \end{fact}

   \begin{proof}
       This is a straightforward consequence of \eqref{Eq: Double counting}.
   \end{proof}

  And in turn this implies: 

   \begin{fact}\label{Fact: Support are commensurable}
   We have
       $\mu_H(H_{A,\rho}) \ll_{d,\tau} \mu_H(H_{B,\lambda\rho}) \ll_{d,\tau} \mu_H(H_{A,\rho})$
   \end{fact}

   \begin{proof}
       By (\ref{Eq: global 2}) and Fact \ref{Fact: Lower bound minimal} we have that 
       \begin{align*}
           \mu(AB) & \geq \frac{\int_H \mu(A_{hh_0^{-1},\rho}B_{h_0,\lambda\rho})dh}{\mu_H(\B(e,\rho))} \\
           & \geq \frac{\int_{H_{A,\rho}} \mu(B_{h_0,\lambda\rho})dh}{\mu_H(\B(e,\rho))} \\
           & \geq \frac{\mu_H(H_{A,\rho})\mu(B_{h_0,\lambda\rho})}{\mu_H(\B(e,\rho))} \\
           & \gg_{d,\tau} \mu(B)\frac{\mu_H(H_{A,\rho})}{\mu_H(H_{B,\lambda\rho})}.
       \end{align*}
       So $$\frac{\mu_H(H_{A,\rho})}{\mu_H(H_{B,\lambda\rho})} \ll_{d,\tau} \frac{\mu(AB)}{\mu(B)}\ll_{d,\tau}1.$$ The other inequality is obtained in a symmetric way.
   \end{proof}

   \begin{lemma}\label{Lemma: Local measure is almost constant}
Let $\alpha > 0$. There is $c \gg_{d,\tau}1$ such that for all $h$ in a subset of measure at least $\left(1-\left|\alpha-\epsilon\right|^{-1}(\epsilon + \delta + \rho)^c\right)\mu(H_{A,\rho})$ of $H_{A,\rho}$ we have $$\mu(A_{h, \rho}) \geq (1-\alpha)m(\rho) \text{ and } m(\rho) \geq (1-\alpha)\mu(A)\mu_H(\B_H(e,\rho))M. $$ 
\end{lemma}

\begin{proof}
    Taking the difference between line (\ref{Eq: global 3}) and (\ref{Eq: global 4}) we find
    \begin{align*}
        \int_H \mu(A_{hh_0^{-1},\rho})\left[\left(1 + \frac{\mu(B_{h_0,\lambda\rho})^{1/d} }{\mu(A_{hh_0^{-1},\rho})^{1/d}}\right)^d-\left(1 + \frac{\mu_H(\B(e,\lambda\rho))^{1/d}\mu(B)^{1/d}}{\mu_H(\B(e,\rho))^{1/d}\mu(A)^{1/d}}\right)^d\right]dh \hspace{3cm} \\
        \ll_{d,\tau} (\epsilon +\rho + \delta)^{c}\mu_H\left(\B\left(e,(1+\lambda)\rho + O_{d,\tau}(\delta^2)\right)\right)\mu(A)
    \end{align*}
    for some $c \gg_{d,\tau}1$. By the mean value theorem 
    \begin{align}
        \left|\left(1 + \frac{\mu(B_{h_0,\lambda\rho})^{1/d} }{\mu(A_{hh_0^{-1},\rho})^{1/d}}\right)^d-\left(1 + \frac{\mu_H(\B(e,\lambda\rho))^{1/d}\mu(B)^{1/d}}{\mu_H(\B(e,\rho))^{1/d}\mu(A)^{1/d}}\right)^d\right| \gg_{d,\tau}   \left|\frac{\mu(B_{h_0,\lambda\rho}) }{\mu(A_{hh_0^{-1},\rho})} - \frac{\mu_H(\B(e,\lambda\rho))\mu(B)}{\mu_H(\B(e,\rho))\mu(A)}\right|\label{Eq: Application mean value}
    \end{align}
    Let $X$ be the set of all $h \in H$ such that $\mu(A_{hh_0^{-1},\rho}) \geq (1-\alpha)M\mu_H(\B_H(e, \rho)) \mu(A)$. For any $h \notin X$, we have 
    $$\left|\frac{\mu(B_{h_0,\lambda\rho}) }{\mu(A_{hh_0^{-1},\rho})} - \frac{\mu_H(\B(e,\lambda\rho))\mu(B)}{\mu_H(\B(e,\rho))\mu(A)}\right| \geq \frac{\left|\alpha-\epsilon\right|}2\frac{\mu(B_{h_0,\lambda\rho})}{\mu(A_{hh_0^{-1},\rho})}.$$

    Therefore, we deduce from (\ref{Eq: Application mean value}) that 
    $$\left|\left(1 + \frac{\mu_H(\B(e,\lambda\rho))^{1/d}\mu(B)^{1/d}}{\mu_H(\B(e,\rho))^{1/d}\mu(A)^{1/d}}\right)^d - \left(1 + \frac{\mu(B_{h_0,\lambda\rho})^{1/d} }{\mu(A_{hh_0^{-1},\rho})^{1/d}}\right)^d\right|\gg_{d,\tau} \frac{\left|\alpha-\epsilon\right|}2\frac{\mu(B_{h_0,\lambda\rho})}{\mu(A_{hh_0^{-1},\rho})}.$$
    Hence, 
      $$ \left|\alpha-\epsilon\right| \mu(B_{h_0,\rho})\mu_H(X) \ll_{d,\tau} (\epsilon +\rho + \delta)^c\mu_H(\B(e,(1+\lambda)\rho+ O_{d,\tau}(\delta^2)))\mu(A)$$
      Which yields
      \begin{align*}
           \mu_H(X) &\ll_{d,\tau} \left|\alpha-\epsilon\right|^{-1} (\epsilon +\rho + \delta)^c\frac{\mu_H(\B(e,(1+\lambda)\rho+ O_{d,\tau}(\delta^2)))\mu(A)}{\mu(B_{h_0,\rho})} \\
           &\ll_{d,\tau}\left|\alpha-\epsilon\right|^{-1} (\epsilon +\rho + \delta)^c M^{-1} \\
           &\ll_{d,\tau} \left|\alpha-\epsilon\right|^{-1} (\epsilon +\rho + \delta)^c\mu(H_{B,\lambda\rho}) \\
           & \ll_{d,\tau} \left|\alpha-\epsilon\right|^{-1} (\epsilon +\rho + \delta)^c\mu(H_{A,\rho}).
      \end{align*}
    Finally, since $\frac{m(\rho)}{\mu_H(\B_H(e,\rho))\mu(A)} \leq M$ we obtain the conclusion.

\end{proof}

Now, we will deduce Proposition \ref{Proposition: Local mass is constant} as a consequence of Kemperman's inequality in connected compact groups:

\begin{proof}[Proof of Proposition \ref{Proposition: Local mass is constant}.]
Let $\alpha > 0$ to be chosen later. We use the notation introduced in the statement of Proposition \ref{Proposition: Local mass is constant}. For any $h_1 \in H_{A,\rho}^{\alpha}, h_2 \in H_{B,\lambda\rho}$ we have 
$$
    \mu\left((AB)_{h_1h_2, (1+c) \rho +  O_{d,\tau}(\delta^2)}\right) \geq \mu(A_{h_1,\rho}B_{h_2,\lambda\rho}) 
    \geq \mu(A_{h_1,\rho})
    \geq (1-\alpha)m(\rho).
$$

So \eqref{Eq: global 1} and Lemma \ref{Lemma: Local measure is almost constant} imply
\begin{align*}
    \mu_H(\B_H(h, (1+\lambda)\rho + O_{d,\tau}(\delta^2)))\mu(AB) 
    \geq \int_{H^\alpha_{A,\rho}h_0} \mu((AB)_{h,(1+\lambda)\rho + O}) dh + \int_{H\setminus (H^\alpha_{A,\rho}h_0)} \mu((AB)_{h,(1+\lambda)\rho+ O_{d,\tau}(\delta^2)}) dh \\
    \geq (1 - O_{d,\tau}(\epsilon + \delta + \rho)^c) \mu_H(\B_H(h, (1+\lambda)\rho+ O_{d,\tau}(\delta^2)))\mu(AB) + \mu_H(H^{\alpha}_{A,\rho}H_{B,\lambda\rho} \setminus H^{\alpha}_{A,\rho}h_0)\frac{m(\rho)}{2} 
\end{align*}

So 

$$\mu_H\left(H^{\alpha}_{A,\rho}H_{B,\lambda\rho} \setminus H^{\alpha}_{A,\rho}h_0\right)\frac{m(\rho)}{2}  \ll_{d,\tau} (\epsilon + \delta + \rho)^c \mu_H(\B_H(h, (1+\lambda)\rho+O_{d,\tau}(\delta^2))\mu(AB).$$
Which implies 
$$ \mu_H(H^{\alpha}_{A,\rho}H_{B,\lambda\rho} \setminus H^{\alpha}_{A,\rho}h_0) \ll_{d,\tau} (\epsilon + \delta + \rho)^c\mu_H(H_{A,\rho})$$
by Fact \ref{Fact: Lower bound minimal}. By Kemperman's inequality,
$$\mu_H(H^{\alpha}_{A,\rho}H_{B,\lambda\rho} ) \geq \min\left(\mu_H(H^\alpha_{A,\rho})+\mu_H(H_{B,\lambda\rho}),1\right).$$
So
$$ \min( 1- \mu_H(H^\alpha_{A,\rho}), \mu_H(H_{B,\lambda\rho})) \ll_{d,\tau} (\epsilon + \delta + \rho)^c\mu_H(H_{A,\rho}).$$
According to Fact \ref{Fact: Support are commensurable} and Lemma \ref{Lemma: Local measure is almost constant}, we find $1-\mu_H(H_{A,\rho}) \ll_{d,\tau}\left|\alpha-\epsilon\right|^{-1}(\epsilon + \delta + \rho)^c$. Choosing $\alpha := (\epsilon + \delta + \rho)^{c/2}$ concludes.

It remains to prove that a similar statement holds for $B$. Note that because of the first part of the proof, there is $h \in H$ such that $\frac{\mu(A_{h,\rho})}{\mu(A)} \geq (1- (\epsilon + \delta + \rho)^{c/2}) M$ as soon as $\epsilon, \rho, \delta$ are sufficiently small depending on $d, \tau$ alone. We are therefore in a position to run a symmetric argument with $(\epsilon + \delta + \rho)^{c/2}$ playing the role of $\epsilon$.
\end{proof}

   \subsection{Reducing to subsets in the neighbourhood of a closed subgroup}
We finally perform one last reduction in preparation of the proof of the main theorem (Theorem \ref{Theorem: Global stability}). Throughout, we will work with two subsets satisfying an almost-equal condition: 

\begin{assumption}\label{Assumption Global}
    Let $A,B \subset G$ and $\epsilon > 0$.  We assume that 
    $$ \mu(A) < \epsilon, \ \mu(B) < \epsilon$$
    and 
    $$\mu(AB) \leq (1+\epsilon) \left(\mu(A)^{1/d'} + \mu(B)^{1/d'} \right)^{d'}$$
    where $d'$ denotes the minimal codimension of a proper closed subgroup.
\end{assumption}

   As in \cite[\S 7]{Machado2024MinDoubling}, the first part of the proof is to show that $A$ and $B$ are contained in precisely one translate of a neighbourhood of a subgroup of maximal dimension. 
   
   \begin{lemma}\label{Lemma: last technical lemma}
    Let $\delta > 0$. Let $A,B$ and $\epsilon$ satisfy Assumption \ref{Assumption Global}. If $\epsilon > 0$ is sufficiently small depending on $\tau, d, \delta$ alone, then there is a proper closed subgroup subgroup $H$ and $g_1, g_2 \in G$ such that: 
    \begin{enumerate}
        \item $A \subset g_1 H_{\delta}$ and $B \subset H_\delta g_2$;
        \item Define the subset $X_\alpha$ made of all $h \in H$ such that $$\left|\frac{\mu(A_{h,\delta})}{\mu(A)\mu_H(\B_H(e,\delta))} - 1 \right| > \alpha \text{ and }\left|\frac{\mu(B_{h,\lambda\delta})}{\mu(B)\mu_H(\B_H(e,\delta))} - 1 \right| > \alpha.$$ Then there is $c \gg_{d,\tau} 1$ such that $\mu_H(X_{(\epsilon + \delta)^c}) \ll_{d,\tau} (\epsilon+\delta)^c$.
    \end{enumerate}
   \end{lemma}

    \begin{proof}
       For all $\delta > 0$, we know by \cite[Prop. 1.6]{Machado2024MinDoubling} and \cite[Thm. 4.6]{TaoProductSet08} that for all $\epsilon \ll_{d,\tau,\delta} 1$ there are there are finite subsets $F_1,F_2$ with $|F_1|,|F_2| = O_{d,\tau}(1)$  such that $A \subset F_1 H_{\delta}$ and $B \subset H_{\delta}F_2$.  By Lemma \ref{Lemma: Locality}, we can moreover assume that $|F_1|=|F_2|$ and both families $(fH_{100\delta})_{f\in F_1}$ and $(H_{100\delta} f)_{f\in F_2}$ are made of pairwise disjoint subsets. 

        Write $M:=\max_{f_1 \in F_1, f_2 \in F_2} \{\frac{\mu(A \cap f_1H_\delta)}{\mu(A)},\frac{\mu(B \cap H_\delta f_2)}{\mu(B)}  \}$. Suppose as we may that there is $f_0 \in F_2$ such that $\frac{\mu(B \cap H_\delta f_0)}{\mu(B)} = M$. In particular, 
        \begin{equation}
            \frac{\mu(B \cap H_\delta f_0)}{\mu(A \cap fH_\delta)} \geq \frac{\mu(B)}{\mu(A)}
        \end{equation}
        for all $f \in F_1$.

        By Lemma \ref{Lemma: Locality} applied to $A$ and $B$, as soon as $\epsilon, \delta$ are sufficiently small depending on $d,\tau$ alone, for all $f \in F_1$, 
        $$\mu(A \cap fH_\delta) \gg_{d,\tau}M \gg_{d,\tau} \mu(AB)$$ and for all $f' \in F_2$, 
        $$\mu(B \cap H_\delta f') \gg_{d,\tau}M\gg_{d,\tau} \mu(AB).$$

        Now,  $$ A \left(B \cap H_\delta f_0\right) \subset AB$$ and 
        \begin{align}
             \mu\left(A \left(B \cap H_\delta f_0\right) \right) 
                    & =\mu\left(\bigsqcup_{f \in F_1} (A \cap fH_\delta)\left(B \cap H_\delta f_0\right) \right) \\
                    & = \sum_{f \in F_1} \mu\left((A \cap fH_\delta)(B \cap H_\delta f_0)\right) \\
                    & \geq \sum_{f \in F_1}(1- O_{d,\tau}(\delta^2))\left(\mu(A \cap fH_\delta)^{1/d} + \mu(B \cap H_\delta f_0)^{1/d}\right)^d \\
                    & \geq (1-O_{d,\tau}(\delta^2)) \left(\mu(A)^{1/d} + \mu(B)^{1/d}\right)^d \\
                    & \geq (1-O_{d,\tau}(\delta^2))\mu(AB).
        \end{align}
        From these inequalities we draw two consequences. 

        Firt of all, for all $f \in F_1$, 
        \begin{align*}
            (1- O_{d,\tau}(\delta^2))\left(\mu(A \cap fH_\delta)^{1/d} + \mu(B \cap H_\delta f_2)^{1/d}\right)^d &\leq \mu\left((A \cap fH_\delta)(B \cap H_\delta f_2)\right) + O_{d,\tau}(\delta+\epsilon)^{c} \mu(AB) \\
            &\leq \mu\left((A \cap fH_\delta)(B \cap H_\delta f_2)\right) + O_{d,\tau}(\delta+\epsilon)^c \mu(A \cap fH_\delta).
        \end{align*}
        We can thus apply Proposition \ref{Proposition: Local mass is constant} to $A \cap fH_\delta$ and $B \cap H_\delta f_0$. 

        Moreover, for every $f \in F_1$, $f' \in F_2$, 
        $$ \mu\left((A \cap fH_\delta)(B \cap H_\delta f') \setminus \bigcup_{f''\in F_1} f''H_{2\delta}f_0\right) \ll_{d,\tau} (\delta + \epsilon)^c\mu(A).$$
        So $$ \mu\left((A \cap fH_\delta) \setminus \bigcup_{f''\in F_1} f''H_{2\delta}f_0f'^{-1}\right) \ll_{d,\tau} (\delta + \epsilon)^c\mu(A).$$
        Now, Proposition \ref{Proposition: Local mass is constant} implies the following estimate for $1\gg_{d,\tau}\delta,\epsilon > 0$, 
        $$ \mu\left((A \cap fH_\delta) \setminus \bigcup_{f''\in F_1} f''H_{2\delta}f_0f'^{-1}\right) \gg_{d,\tau} \mu(A)\mu_H\left(H \setminus\bigcup_{f''\in F_1} f^{-1}f''H_{2\delta}f_0f'^{-1} \right).$$
        So for $\delta, \epsilon \ll_{d,\tau}1$, 
        $$\mu_H\left(H \setminus\bigcup_{f''\in F_1} f''H_{2\delta}f_0f'^{-1} \right) \leq \frac{1}{100}.$$
        Thus, there is $f'' \in F_1$ such that 
        $$ \mu_H(H \cap f'' H_{2\delta}f_0 f'^{-1}) \gg_{d,\tau}1.$$
        We get 
        $$ \mu_H(H \cap (f_0 f'^{-1})^{-1} H_{2\delta}f_0 f'^{-1}) \gg_{d,\tau}1.$$
        And by Kemperman's inequality \cite{Kemperman64}, we find that for $m \ll_{d,\tau} 1$, 
        $$ \mu_H\left(H \cap (f_0 f'^{-1})^{-1} H_{2m\delta}f_0 f'^{-1}\right) \geq \mu_H\left(\left(H \cap (f_0 f'^{-1})^{-1} H_{2\delta}f_0 f'^{-1}\right)^m\right) = 1.$$

        So $H \subset (f_0f'^{-1})^{-1} H_{O_{d,\tau}(\delta)}f_0 f'^{-1}$. Hence, $f_0f'^{-1} \in N(H)_{O_{d,\tau}(\delta)}$ where $N(H)$ denotes the normaliser of $H$. A symmetric argument shows that $f^{-1}f' \in N(H)_{O_{d,\tau}(\delta)}$ for all $f,f' \in F_1$. Thus, there are $f_1 \in F_1$, $f_2 \in F_2$, two subsets $\tilde{F_1},\tilde{F_2} \subset N(H)$ and $\delta'\ll_{d,\tau}\delta$ such that $A \subset f_1\tilde{F_1}H_{\delta'}$ and $B \subset H_{\delta'}\tilde{F_2}f_2$. Upon considering $f_1^{-1}A$ and $Bf_2^{-1}$ instead of $A$ and $B$ respectively, we may assume that $f_1 =f_2 =e$. We can apply Lemma \ref{Lemma: Locality} once more and assume that $|\tilde{F_1}| = |\tilde{F_2}|$ and for all $f_1 \in \tilde{F_1}, f_2 \in \tilde{F_2}$ and 
        \begin{equation}
            \frac{\mu(A \cap f_1H_{\delta'})}{\mu(A)} = (1+O_{d,\tau}(\epsilon + \delta)^c)\frac{\mu(B \cap H_{\delta'}f_2)}{\mu(B)} \label{Eq: Equal mass on each coset}
        \end{equation}

       Let $\pi: H \rightarrow N(H)/H$ denote the canonical projection. Suppose that $|\pi(\tilde{F_1}\tilde{F_2})| > |\pi(\tilde{F_1})|$ and fix $f_2 \in \tilde{F_2}$. Take now $f \in \tilde{F_1}$, $f' \in \tilde{F_2}$ such that $\pi(ff') \notin \pi(\tilde{F_1}f_2)$. Recall that any two connected subgroups of minimal co-dimension are conjugate. Hence, for $f \in N(H) \setminus H$, $d(fH,H) \gg_d 1$. Thus, for $\delta$ - hence, $\delta'$ - sufficiently small depending on $\tau, d$, $$fH_{2\delta'}f' \cap \bigcup_{f'' \in F_1}f''H_{2\delta'}f_2 =\emptyset. $$
        So 
        $$ \mu(AB) \geq \mu(A(B \cap H_\delta' f_2)) + \mu\left((A \cap fH_{\delta'})(B \cap H_{\delta'}f')\right).$$
        Since $\mu((A \cap fH_{\delta}'))\gg_{d,\tau} \mu(AB)$, this yields a contradiction. So $|\pi(\tilde{F_1}\tilde{F_2})| = |\pi(\tilde{F_1})|$. Since $|\pi(\tilde{F_1})|=|\tilde{F_1}|=|\tilde{F_2}|=|\pi(\tilde{F_2})|$ and $e \in \tilde{F_1}$, $e \in \tilde{F_2}$, this means that $\pi(\tilde{F_1})=\pi(\tilde{F_2})$ is a subgroup of $N(H)/H$. So (1) is proved. Now, (2) follows from (\ref{Eq: Equal mass on each coset}) and Proposition \ref{Proposition: Local mass is constant}.
        \end{proof}




\subsection{Concluding the proof}\label{Subsection: Concluding}
We are finally ready to conclude the proof of Theorem \ref{Theorem: Global stability}. We will use the same argument twice in a row - this argument is close in spirit to ones found in the proof of stability results in \cite{jing2023measure2,Machado2024MinDoubling}. They all rely on stability properties of group homomorphisms between compact groups, see \cite{zbMATH03429759, KazhdaneRep82}. 

First of all, we have to prepare the grounds to use the local stability result. In particular, it is necessary to show that the assumption about convex hull is respected locally.

\begin{lemma}\label{Lemma: local convex hull is large}
Let $A,B$ and $\epsilon > 0$ satisfy Assumption \ref{Assumption Global}. Let, moreover, $H \subset G$ be a closed subgroup and $\delta > 0$ be such that $g_1A,Bg_2 \subset H_\delta$ for some $g_1,g_2 \in G$ and the conclusions of Lemma \ref{Lemma: last technical lemma} are satisfied.

Then, we can choose $H$ and $\delta$ in such a way that:
\begin{enumerate}
    \item For all $(h_1,h_2)$ in a subset of measure at least $1 - O_{d,\tau}(\epsilon + \delta)^c$ of $H \times H$, we have 
    $$ \mu(A_{h_1,\delta}B_{h_2,\lambda\delta}) \leq \mu((AB)_{h_1h_2,(1+\lambda)\delta + O_{d,\tau}(\delta^2)})  \leq (1 + ( \epsilon + \delta)^c)\left(\mu(A_{h_1,\delta})^{1/d} + \mu(B_{h_2,\lambda\delta})^{1/d}\right)^d;$$
    \item For all $h$ in a subset of measure at least $1 - O_{d,\tau}(\epsilon + \delta)^c$ of $H$, we have 
    $$ \mu\left(\co( A_{h,\delta})\right) \gg_{d,\tau}\mu(\B(e,\delta)) \text{ and } \mu\left(\co( B_{h,\lambda\delta})\right) \gg_{d,\tau}\mu(\B(e,\delta)).$$
\end{enumerate}
\end{lemma}

\begin{proof}
    The letter $c$ will denote constants that might change from line to line but will all satisfy $c \gg_{d,\tau}1$. For technical reasons, working with scales slightly below $\delta$ will be convenient in this proof. Fix $\frac32 \geq \beta \geq 1$. Choose $h \in H$ such that $\left|\mu(B_{h,\lambda\delta^\beta}) - \mu_H(\B_H(e,\lambda\delta^\beta))\mu(B)\right| \ll_{d,\tau} (\epsilon + \delta)^c\mu_H(\B_H(e,\lambda\delta^\beta))\mu(B)$. By Lemma \ref{Lemma: Local measure is almost constant},  the subset of those $h$'s has measure $1-O_{d,\tau}\left(\epsilon + \delta\right)^c$.

    Subsrtracting \eqref{Eq: global 3} from \eqref{Eq: global 2}, 

    \begin{align}
        \int_H \left|\mu((AB)_{h'h,(1+\lambda)\delta^\beta + O_{d,\tau}(\delta^2)}) - (1-O_{d,\tau}(\delta)^c)\left(\mu(A_{h',\delta^\beta})^{1/d} + \mu(B_{h,\lambda\delta^\beta})^{1/d}\right)^d\right| dh'\hspace{2cm} \\ \ll_{d,\tau} (\epsilon + \delta)^c\mu(AB)\mu_H(\B_H(e,\delta^\beta)). \notag
    \end{align}
    By Markov's inequality,  for all $h'$ in a subset of $H$ of measure at least $1-O_{d,\tau}\left(\epsilon + \delta\right)^c$, 
    $$ \mu(A_{h',\delta^\beta}B_{h,\lambda\delta^\beta}) \leq \mu((AB)_{h'h,(1+\lambda)\delta^\beta + O_{d,\tau}(\delta^2)})  \leq (1 + O_{d,\tau}( \epsilon + \delta)^c)\left(\mu(A_{h',\delta^\beta})^{1/d} + \mu(B_{h,\lambda\delta^\beta})^{1/d}\right)^d$$
    and, by Proposition \ref{Proposition: Local mass is constant},
    $$ \left|\mu(A_{h', \delta^\beta}) - \mu_H(\B(e, \delta^\beta))\mu(A)\right| \ll_{d,\tau} (\epsilon + \delta)^c\mu_H(\B(e, \delta^\beta))\mu(A).$$

    Set $Y \subset  H \times H$ the subset of all such pairs $(h,h')$. We have shown that $(\mu_H \times \mu_H)(Y) \geq 1- O_{d,\tau}( \epsilon + \delta)^c$. When $\beta = 1$,  this proves (1).

    We will now show (2). We will now use that case $\beta = \frac32$. For all $(h,h') \in Y$, Lemma \ref{Lemma: Locality} applied to $A_{h',\delta^\beta}, B_{h,\lambda\delta^\beta}$ and $C$ the convex hull $\co(B_{h,\lambda\delta^\beta})$, yields $A_{h',\delta^\beta} \subset g \left(C\right)^m$ for some $m=O_{d,\tau}(1)$ and $g \in G$.  Applying Lemma \ref{Lemma: Locality} once more, we see that for any $h'' \in H$ such that $(h'',h') \in Y$, we have $B_{h'',\lambda\delta^\beta} \subset C^{m^2}r$ for some $r \in G$. By a Fubini-type argument, we see that for a proportion $1 - O_{d,\tau}( \epsilon + \delta)^c$ of all $h \in G$ there are a subset $Y_h \subset Y$ with $\mu_H(Y_h) \geq 1 - O_{d,\tau}( \epsilon + \delta)^c$ and $m = O_{d,\tau}(1)$ such that for all $(h'',h') \in Y_h$,
    $$ A_{h', \delta^{\beta}} \subset g(h')C^m \text{ and } B_{h'',\lambda\delta^\beta} \subset C^mr(h'')$$
    where $C=\co(B_{h,\lambda\delta^\beta})$. Note in particular that $g(h'),r(h') \in \B(h', O_{d,\tau}(\delta))$. 
   
   For any such $h$, write $X_h \subset Y_h$ the subset of pairs $(h_1,h_2)$ such that
   $$\mu_H\left(\{(h',h_1h_2h'^{-1}) \in Y_h\}\right) \geq 1- O_{d,\tau}( \epsilon + \delta)^c.$$
   We have that $\left(\mu_H \times \mu_H\right)(X_h) \geq 1 - O_{d,\tau}(\epsilon + \delta)^c$ by a Fubini-type argument.
   

   Take $x \in H$ such that $x=h_1h_2$ with $(h_1,h_2) \in X_h$. Then 
   \begin{align}\mu(A_{xh'^{-1},\delta^\beta}B_{h',\lambda\delta^\beta}) &\geq (1 - O_{d,\tau}( \epsilon + \delta)^c)\left(\mu(A_{xh'^{-1},\delta^\beta})^{1/d} + \mu(B_{h',\lambda\delta^\beta})^{1/d}\right)^d \\ &\geq (1 - O_{d,\tau}( \epsilon + \delta)^c)\mu((AB)_{x,(1+\lambda)\delta^\beta+ O_{d,\tau}(\delta^2)}). \label{Eq: Local doubling}
   \end{align}
    and 
    \begin{equation}
        A_{xh'^{-1},\delta^\beta}B_{h',\lambda\delta^\beta} \subset g(xh'^{-1})C^{2m^2}r(h') \label{Eq: Convex envelope}
    \end{equation}
    for $h'$ ranging in a subset of measure at least $1-O_{d,\tau}( \epsilon + \delta)^c$. In particular, for any $m \geq 0$, take $h'_1,\ldots, h'_m$ as above. We have 
    $$
        \mu\left( (AB)_{x,(1+\lambda)\delta^\beta+ O_{d,\tau}(\delta^2)} \cap \bigcap_{i=1}^m g(xh'^{-1}_i)C^{2m^2}r(h'_i) \right) \geq (1-O_{d,\tau,m}( \epsilon + \delta)^c)\mu((AB)_{x,(1+\lambda)\delta^\beta + O_{d,\tau}(\delta^2)}).$$
    and so for any $g \in \bigcap_{i=1}^m g(xh'^{-1}_i)C^{2m^2}r(h'_i)$ we have

       \begin{equation}
        \mu\left( (AB)_{x,(1+\lambda)\delta^\beta+ O_{d,\tau}(\delta^2)} \cap g\left(\bigcap_{i=1}^m r(h'_i)^{-1}C^{4m^2}r(h'_i) \right)\right) \geq (1-O_{d,\tau,m}( \epsilon + \delta)^c)\mu((AB)_{x,(1+\lambda)\delta^\beta}). \label{Eq: Large intersection with a convex subset}
    \end{equation}
    
    We will exploit the following simple claim: 
    
    \begin{claim}\label{Claim: Small axis implies small convex hull}
        There is $\eta >0$ such that the following holds. Let $C \subset \B(e,\rho)$ be a convex subset and $X \subset H$ be a subset such that $\mu(X\B(e,\eta))\geq 1-\eta$. Suppose that $\rho$ is sufficiently small that the exponential and log maps are well-defined diffeomorphisms. Let $C_{H^\perp}$ denote the convex subset obtained by orthogonally projecting $C$ to the subspace orthogonal to the Lie algebra of $H$. Call $H^\perp$-axis of $C$ the axis of $C_H$. If one $H^\perp$-axis of has length $l$, then there is a subset $F \subset X$ of size $O_d(1)$ such that the projection of 
        $$ \bigcap_{h\in F} hCh^{-1}$$
        to $H^\perp$ is contained in a ball of radius $O_{d}(l)$. 
    \end{claim}

    \begin{proof}
        This is clear from Lemma \ref{Lemma: Symmetric short axis implies small volume} and the fact that $H$ acts irreducibly on $H^\perp$, see \cite[App. A]{Machado2024MinDoubling}.
    \end{proof}

    By Claim \ref{Claim: Small axis implies small convex hull} combined with \eqref{Eq: Large intersection with a convex subset},
    $$\mu\left((AB)_{x,(1+\lambda)\delta^\beta+ O_{d,\tau}(\delta^2)} \cap \B(g,O_{d,\tau}(l))\right) \geq (1-O_{d,\tau}( \epsilon + \delta)^c) \mu((AB)_{x,(1+\lambda)\delta^\beta})$$ for some $g \in G$ - where $l$ denotes the maximum between the length of the $H^\perp$-shortest axis of $C$ and $\delta^{\frac32}$.

    Since balls are normal,  $$A_{xh'^{-1},\delta^\beta} \subset \B(g(xh'^{-1}), O_{d,\tau}(l)) \text{ and } B_{h',\lambda\delta^\beta} \subset \B(r(h'),O_{d,\tau}(l)) $$
    for some $g(xh'^{-1}),r(h') \in G$. In addition, as soon as $\epsilon,  \delta \ll_{d,\tau}1$, $$A_{hh'^{-1},\delta^\beta}B_{h',\lambda\delta^\beta} \cap A_{hh''^{-1},\delta^\beta}B_{h'',\lambda\delta^\beta} \neq \emptyset$$ for any $(xh''^{-1},h'') \in X_h$ by \eqref{Eq: Local doubling}. So \begin{equation}g(hh'^{-1})r(h') \in \B\left(g(hh''^{-1})r(h''), O_{d,\tau}(l)\right). \label{Eq: Indep definition b}\end{equation}

    Define now the map $b: H \rightarrow G$ by $b(x) = g(h_1)r(h_2)$ for some choice of $h_1,h_2 \in X_h$ such that $h_1h_2=x$. Note that if $h_3, h_4$ were any other choice, $b(x) \in g(h_3)r(h_4) \B_H(e,O_{d,\tau}(l))$ by (\ref{Eq: Indep definition b}). 
    
    \begin{claim}\label{Claim: almost rep}
    The map $h \mapsto b(h)b(e)^{-1}$ is $O_{d,\tau}(l)$-close to a group homomorphism.
    \end{claim}

    \begin{proof}
        Note that $b$ is well-defined on a subset of $H$ of measure $1- O_{d,\tau}(\epsilon + \delta)^c$. For all $h \in H$, define $\phi(h):=b(h_1)b(h_2)^{-1}$ for some $h_1, h_2$ such that $b$ is well defined, $h_1h_2^{-1}=h$ and both $(h_1,h_1^{-1}), (h_2,h_2^{-1}) \in X_h$. For $\epsilon, \delta \ll_{d,\tau} 1$, $b$ is well-defined over $H$. Moreover, if $h_3, h_4$ were any possible choice, we would have $b(h_3)b(h_4)^{-1} \in \B_H(\phi(h), O_{d,\tau}(l))$. Indeed, for $ \epsilon, \delta$ sufficiently small, we can find $g_1, g_2, g_3, g_4 \in H$ such that $g_2=g_4$, $g_1g_2 = h_3$, $g_3g_4 = h_4$ and $g(g_1), r(g_2), g(g_4)$ and $r(g_3)$ are well-defined. Thus,

        $$ b(h_3)b(h_4)^{-1} \in g(g_1)r(g_2) r(g_4)^{-1}g(g_3)^{-1}\B_H(e,O_{d,\tau}(l)) = g(g_1)g(g_3)^{-1}\B_H(e,O_{d,\tau}(l))$$
        according to (\ref{Eq: Indep definition b}). Moreover, looking at the conditions on $g_1, g_2,g_3,g_4$, we see that $g_1g_3^{-1}=h$ and $g_3$ can be freely chosen in a subset of measure $1- O_{d,\tau}( \epsilon + \delta)^c$. Since this is true for any choice of $h_3, h_4$, including $h_3=h_1, h_4=h_2$, we have, 
        $$ \phi(h) \in b(h_3)b(h_4)^{-1}\B_H(e,O_{d,\tau}(l))$$
        as soon as $ \delta$ and $\epsilon$ are sufficiently small depending on $\tau$ and $d$ alone.
         
        Now, we wish to show that it satisfies 
        $$ d(\phi(x)\phi(y), \phi(xy)) \ll_{d,\tau} l$$
        for all $x,y \in H$. Indeed, as above, we can find $h_1, h_2,h_3$ such that $x=h_1h_2^{-1}$, $y= h_2h_3^{-1}$ and $b$ is defined on $h_1, h_2$ and $h_3$. Then,
        \begin{align*}
            \phi(x)\phi(y) &\in (b(h_1)b(h_2)^{-1})(b(h_2)b(h_3)^{-1})\B_H(e,O_{d,\tau}(l)) \\
            & \in b(h_1)b(h_3)^{-1}\B_H(e,O_{d,\tau}(l)) \\
            & \in \phi(xy)\B_H(e,O_{d,\tau}(l)).
        \end{align*}
        So $\phi$ - and, hence, $h \mapsto b(h)b(e)^{-1}$ - is within distance $O_{d,\tau}(l)$ of an actual homomorphism $\phi: H \rightarrow G$, see \cite{zbMATH03429759}. 
    \end{proof}
    Hence, $\mu(AB \cap \phi(H)_{O_{d,\tau}(l)}b(e))\geq \frac23 \mu(AB)$. And by Corollary \ref{Corollary: Locality 99 percent} again, there are $g_1, g_2 \in G$, $b \in G$ such that $A \subset g_1\left(b^{-1}\phi(H)b\right)_{O_{d,\tau}(l)}$ and $B \subset \left(b^{-1}\phi(H)b\right)_{O_{d,\tau}(l)}g_2$. 

    We are finally ready to conclude (2). Suppose as we may that we had chosen $(H,\delta)$ with $A \subset g_1H_\delta$ and $B \subset H_\delta g_2$ for some $g_1,g_2 \in G$ in such a way as to minimise the value of $\delta$. The above proof then shows that for any $h \in H$ such that $(\mu_H \times \mu_H)(Y_h) \geq 1 - O_{d,\tau}(\epsilon + \delta)^c$, we have that all $H^\perp$-axes of $\co(B_{h,\delta^\beta})$ have length $\gg_{d,\tau} \delta$. This readily implies that for all $h$ in a subset of measure at least $1 - O_{d,\tau}(\epsilon + \delta)^c$  we have $\mu(\co(B_{h,\delta})) \gg_{d,\tau} \mu(T_\delta\B_H(e,\delta))$. This proves (2).
    
\end{proof}

    We can conclude the proof of the main result with a strikingly similar line of reasoning. 
    
    \begin{proof}[Proof of Theorem \ref{Theorem: Global stability}.]
         Let $H$ and $\delta$ be given by Lemma \ref{Lemma: last technical lemma}. Upon considering  $gA, Bh$ instead of $A,B$, we may assume the $A,B \subset H_\delta$.  By Lemma \ref{Lemma: local convex hull is large}, there is $c \gg_{d,\tau}1$ such that the subset $Y \subset H \times H$ made of those $(h_1,h_2)$ such that:
         \begin{enumerate}
            \item $\mu(A_{h_1,\delta}) = (1 + O_{d,\tau}(\epsilon + \delta)^c)\mu(A) \mu(\B(e,\rho))$ and $\mu(B_{h_2,\lambda\delta}) == (1 + O_{d,\tau}(\epsilon + \delta)^c)\mu(B) \mu(\B(e,\lambda\rho))$;
             \item $ \mu(A_{h_1,\delta}B_{h_2,\lambda\delta}) \leq \mu((AB)_{h_1h_2,(1+\lambda)\delta + O_{d,\tau}(\delta^2)})  \leq (1 + ( \epsilon + \delta)^c)\left(\mu(A_{h_1,\delta})^{1/d} + \mu(B_{h_2,\lambda\delta})^{1/d}\right)^d;$
             \item
                 $\mu\left(\co( A_{h_1,\delta})\right) \gg_{d,\tau}\mu(\B(e,\delta)) \text{ and } \mu\left(\co( B_{h_2,\lambda\delta})\right) \gg_{d,\tau}\mu(\B(e,\delta))$;
             
             \item for every $(h_1,h_2) \in Y$ the subset $\{h \in H : (h_1,h) \in Y\}$ (resp. $\{h \in H : (h,h_2) \in Y\}$) has measure at least $1-O_{d,\tau}(\epsilon+\delta)^c$;
             \item for every $(h_1,h_2) \in Y$ the subset $\{h \in H : (h_1h_2h^{-1},h) \in Y\}$  has measure at least $1-O_{d,\tau}(\epsilon+\delta)^c$;
             \item upon considering $hA, Bh'$ for some $h,h' \in H$, we may assume that $(e,e) \in Y$. 
         \end{enumerate}
         has measure $1- O_{d,\tau}( \epsilon + \delta)^c$. We can ensure (1)-(3) because of Lemma \ref{Lemma: local convex hull is large} and (4) and (5) because of a Fubini-type argument. By the stability result for the local Brunn--Minkowski (Theorem \ref{Theorem: Local stability}), for every $(h_1, h_2) \in Y$ we have that there are a convex subset $C_{h_1} \subset \B(e,\delta)$ with center of mass $e$ and $g \B(h_1,2\delta), r \in \B(h_2,2\delta)$ such that  
         $$\mu\left(gC_{h_1}\Delta A_{h_1,\delta}\right) \ll_{d,\tau}  (\epsilon + \delta)^c\mu\left( A_{h_1,\delta}\right), \hspace{0.5cm} \mu\left((\lambda C_{h_1})r\Delta B_{h_2,\lambda_2\delta}\right) \ll_{d,\tau}  (\epsilon + \delta)^c\mu\left(C_{h_1}\right)$$
         and 
         $$ \mu\left((AB)_{h_1h_2,(1+\lambda)\delta+ O_{d,\tau}(\delta^2)} \Delta g\left((1+\lambda)C_{h_1}\right)r\right)  \ll_{d,\tau}  (\epsilon + \delta)^c\mu\left(C_{h_1}\right)$$
         where, as above, $\lambda=\frac{\mu(B)^{1/d}}{\mu(A)^{1/d}}$.


       Thus, for every $(h_1,h_2), (h_3,h_2) \in Y$ we have 

       $$\mu(C_{h_1} \Delta C_{h_3}) \ll_{d,\tau}  (\epsilon + \delta)^c\mu\left(C_{h_1}\right).$$

       Since for every two $h_1, h_3$ appearing as first coordinate in $Y$, there is $h_2$ such that $(h_1,h_2)$ and $(h_3,h_2)$, we can assume that $C_{h_1}$ is independent of $h_1$ and write $C:=C_{h_1}$. 
In particular, for every $h$ appearing as a first coordinate in $Y$, there is $g(h) \in \B(e,2\delta)$ such that $ \mu(A_{h,\delta} \Delta hg(h)C) \ll_{d,\tau}  (\epsilon + \delta)^c\mu\left(C\right).$ Similarly, for every $h$ appearing as a second coordinate in $Y$, there is $r(h) \in \B(e,2\delta)$ such that $ \mu(A_{h,\delta} \Delta Cr(h)h) \ll_{d,\tau}  (\epsilon + \delta)^c\mu\left(C\right).$ Moreover, for all $(h_1,h_2) \in Y$ we have $$ \mu((AB)_{h_1h_2, (1+\lambda)\delta + O_{d,\tau}(\delta^2)} \Delta h_1g(h_1)\left((1+\lambda)C\right)r(h_2)h_2)\ll_{d,\tau}  (\epsilon + \delta)^c\mu\left(C\right).$$
To simplify the above inequality, notice that the order in the product matters only up to an error of size $O_{d,\tau}  (\epsilon + \delta)^c\mu\left(C\right)$ - because of Corollary \ref{Corollary: Sum vs product} and the inequality $\mu(C) \gg_{d,\tau} \mu(\B(e,\delta))$. Namely,
$$\mu\left[h_1g(h_1)r(h_2)h_2\left((1+\lambda)\left(h_2^{-1}Ch_2\right)\right) \Delta h_1g(h_1)(1+\lambda)Cr(h_2)h_2\right]\ll_{d,\tau}  (\epsilon + \delta)^c\mu\left(C\right)$$
       Now, by assumption, for all $h$ in a subset of measure $1 - O_{d,\tau}(\epsilon + \delta)^c$ we have $(h,h^{-1}) \in Y_{\epsilon}$. So the above implies that whenever $(h_1,h_1^{-1}), (h_2, h_2^{-1}) \in Y$, then 
       $$ \mu\left[\left(h_1g(h_1)r(h_1^{-1})h_1^{-1}\right)h_1Ch_1^{-1} \Delta \left(h_2g(h_2)r(h_2^{-1})h_2^{-1}\right)h_2Ch_2^{-1}\right]\ll_{d,\tau}  (\epsilon + \delta)^c\mu\left(C\right).$$
       Which yields
       $$ \mu\left(h_1^{-1}Ch_1 \Delta h_2^{-1}Ch_2\right)\ll_{d,\tau}  (\epsilon + \delta)^c\mu\left(C\right)$$
       by Lemma \ref{Lemma: center of mass}. Since every $h \in H$ can be written $h_2^{-1}h_1$ for some $h_1, h_2$ as above, 
       $$ \mu\left(C\Delta h^{-1}Ch\right)\ll_{d,\tau}  (\epsilon + \delta)^c\mu\left(C\right)$$
       for all $h \in H$. According to Lemma \ref{Lemma: Small sym diff implies close}, we can therefore assume that $C$ is invariant under the action of $H$. 
       Finally, for any $(h_1,h_2),(h_3,h_4) \in Y$ with $h_1h_2=h_3h_4$ we have 
       $$\mu\left(g(h_1)r(h_2)\left((1+\lambda)C\right) \Delta g(h_3)r(h_4)\left((1+\lambda)C\right)\right) \ll_{d,\tau}  (\epsilon + \delta)^c\mu\left(C\right).$$
       According to Lemma \ref{Lemma: center of mass}, we have 
       $$ d\left(h_1g(h_1)r(h_2)h_2,h_3g(h_3)r(h_4)h_4\right) \ll_{d,\tau} (\epsilon + \delta)^c\delta. $$
       Proceeding as in the end of the proof of Lemma \ref{Lemma: local convex hull is large}, we find a group homomorphism $\phi: H \rightarrow G$ such that $h_1g(h_1)r(h_2)h_2$ is within distance $O_{d,\tau}\left( (\epsilon + \delta)^c\delta\right)$ of $\phi(h_1h_2)g$ for all $(h_1,h_2) \in Y$ and some $g \in G$. There is $h_2$ such that for all $h$ in a subset $X$ of measure $1 - O_{d,\tau}(\epsilon + \delta)^c$ we have $(h,h_2) \in Y$. This yields $g_0 \in H_\delta$ such that $hg(h)g_0$ (resp. $g_0^{-1}r(h)$) is within distance $O_{d,\tau}\left( (\epsilon + \delta)^c\delta\right)$ of $\phi(h)g_0$ (resp. $g_0^{-1}\phi(h)$). Since $g_0 \in H_\delta$, $\mu(g_0^{-1}Cg_0 \Delta C) \ll_{d,\tau} \delta^c$. So writing $H'= g_0^{-1}\phi(H)g_0$ and $C'=g_0^{-1}Cg_0$ we find $g, r \in G$ such that 
       $$ \mu(A \Delta gH'C') \ll_{d,\tau} (\epsilon + \delta)^c \mu(A),$$
       $$ \mu(B \Delta (\lambda C')H'r) \ll_{d,\tau} (\epsilon + \delta)^c \mu(B)$$
       and 
       $$\mu(AB \Delta g(1+\lambda)C'H'r)  \ll_{d,\tau} (\epsilon + \delta)^c \mu(AB).$$
       Finally, since $C'$ is invariant under conjugation by  $H'$, we have $C'H' = H'_{\delta'}$, $\lambda C'H'= H_{\lambda \delta'}$ and $(1+\lambda) C'H'= H_{(1+\lambda)\delta'}$ for some $\delta'$, see \cite[App. A]{Machado2024MinDoubling}.

    \end{proof}


\end{document}